\documentclass[11pt,reqno]{amsart}

\usepackage[T1]{fontenc}
\usepackage[boldsans]{ccfonts}
\usepackage{eulervm}
\renewcommand{\mathbf}{\mathbold}
\usepackage{bbm}

\usepackage{enumerate, amsmath, amsthm, amsfonts, amssymb, %xy,  
mathrsfs, graphicx, paralist} %, ulem
\usepackage[usenames, dvipsnames]{color}
\usepackage[margin=1in]{geometry} 
\usepackage{comment}

\usepackage{mathabx} %This is for \sqbullet
\usepackage{tikz-cd}
\numberwithin{equation}{section}

\newtheorem{Theorem}[equation]{Theorem}
\newtheorem{Proposition}[equation]{Proposition}
\newtheorem{Lemma}[equation]{Lemma}

\newtheorem{Conjecture}[equation]{Conjecture}
\newtheorem{Question}[equation]{Question}

\newtheorem{Claim}[equation]{Claim}

\theoremstyle{definition}
\newtheorem{Remark}[equation]{Remark}

\newtheorem{Example}[equation]{Example}
\newtheorem{eg}[equation]{Example}

\newcommand{\cF}{\mathcal{F}}

\newcommand{\cH}{\mathcal{H}}

\newcommand{\cK}{\mathcal{K}}

\newcommand{\cN}{\mathcal{N}}

\newcommand{\cS}{\mathcal{S}}

\newcommand{\cT}{\mathcal{T}}

\newcommand{\ZZ}{\mathbb{Z}}

\renewcommand{\phi}{\varphi}
\renewcommand{\emptyset}{\varnothing}
\newcommand{\eps}{\varepsilon}

\renewcommand{\tilde}[1]{\widetilde{#1}}

\makeatletter
\def\Ddots{\mathinner{\mkern1mu\raise\p@
\vbox{\kern7\p@\hbox{.}}\mkern2mu
\raise4\p@\hbox{.}\mkern2mu\raise7\p@\hbox{.}\mkern1mu}}
\makeatother

\DeclareMathOperator{\sgn}{sgn}

\newcommand{\Waff}{{W_{\text{aff}}}}
\newcommand{\Inv}{\mathrm{Inv}}
\newcommand{\twomat}[4]{\begin{bmatrix}
        #1 & #2  \\ 
        #3 & #4 
     \end{bmatrix}}

\newcommand{\suchthat}{\mid}

\newcommand{\textif}{\text{ if }}
\newcommand{\textand}{\text{ }\mathrm{and}\text{ }}

\newcommand{\Titscone}{\cT}
\newcommand{\WTits}{W_{\Titscone}}
\newcommand{\id}{\mathrm{id}}
\newcommand{\triv}{\mathrm{triv}}
\newcommand{\sign}{\mathrm{sign}}

\newcommand{\afD}{\widetilde{\Delta }}
\newcommand{\afDp}{\widetilde{\Delta }^+}
\newcommand{\innprod}[2]{\langle #1, #2 \rangle}

\newcommand{\Lo}{{\Lambda _0}}

\newcommand{\Th}[1]{{\Theta_{#1}}}

\begin{document}

\title{Pursuing Coxeter theory for Kac-Moody affine Hecke algebras}
\author{Dinakar Muthiah}
\author{Anna Pusk\'as}
\maketitle

%\anna{Hello.}\dinakar{Hello.}
\newcommand{\aff}{\mathrm{aff}}

\begin{abstract}
  The Kac-Moody affine Hecke algebra $\cH$ was first constructed as the Iwahori-Hecke algebra of a $p$-adic Kac-Moody group by work of Braverman, Kazhdan, and Patnaik, and by work of Bardy-Panse, Gaussent, and Rousseau. Since $\cH$ has a Bernstein presentation, for affine types it is a positive-level variation of Cherednik's double affine Hecke algebra.

  Moreover, as $\cH$ is realized as a convolution algebra, it has an additional ``$T$-basis'' corresponding to indicator functions of double cosets. For classical affine Hecke algebras, this $T$-basis reflects the Coxeter group structure of the affine Weyl group. In the Kac-Moody affine context, the indexing set $W_{\cT}$ for the $T$-basis is no longer a Coxeter group. Nonetheless, $W_{\cT}$ carries some Coxeter-like structures: a Bruhat order, a length function, and a notion of inversion sets.

This paper contains the first steps toward a Coxeter theory for Kac-Moody affine Hecke algebras. We prove three results. The first is a construction of the length function via a representation of $\cH$. The second concerns the support of products in classical affine Hecke algebras. The third is a characterization of length deficits in the Kac-Moody affine setting via inversion sets. Using this characterization, we phrase our support theorem as a precise conjecture for Kac-Moody affine Hecke algebras. Lastly, we give a conjectural definition of a Kac-Moody affine Demazure product via the $q=0$ specialization of $\cH$.

\end{abstract}

\section{Introduction}
\subsection{Affine Weyl groups and Hecke algebras} 

We can think about an affine Weyl group $\Waff$ in two different ways. One the one hand, it is a Coxeter group associated to an affine root datum with notions of reduced expressions, length function, and Bruhat order. On the other hand, viewing the affine root datum as the affinization of a finite root datum, we may view $\Waff$ as the semi-direct product of a finite Weyl group and a lattice. Corresponding to these points of view, the affine Hecke algebra $\cH_{\Waff}$ has two presentations. 

Viewing $\Waff$ as a Coxeter group, one has the Iwahori-Matsumoto presentation and the $T$-basis $\{T_w\}_{w \in \Waff}$ of $\cH_\Waff$ given by taking products of reduced expressions of generators. The generators satisfy  quadratic relations \eqref{eq:T-identities} and additionally braid relations, which can be encoded via the following length additivity formula:
\begin{equation}
  \label{eq:length-additive-intro}
T_x T_y = T_{xy}\text{ if }\ell(x)+\ell(y)=\ell(xy)\text{ for }x,y\in \Waff
\end{equation}
Here $\ell : \Waff \rightarrow \ZZ_{\geq 0}$ is the length function.

The $T$-basis also reflects the realization of $\cH_\Waff$ as the Iwahori-Hecke algebra of reductive group over a non-archimedean local field \cite{iwahori1965some} (for brevity, we will refer to such a group as a $p$-adic group). In particular, the $T$-basis is equal to the basis consisting of indicator functions of Iwahori double cosets and the length function on $\Waff$ exactly encodes the Haar measure of the double cosets. 

Viewing $\Waff$ as the semi-direct product of a finite Weyl group and a lattice, we have the Bernstein presentation of $\cH_\Waff.$ The algebra $\cH_\Waff$ is generated by the finite Hecke algebra and the group algebra of the lattice. The interaction of these two subalgebras is controlled by the Bernstein relation \eqref{eq:Bernstein}.

\subsection{The Kac-Moody affine setting}

In this paper we study the \emph{Kac-Moody affine Hecke algebra} $\cH$, which arises from affinizing a Kac-Moody root datum. The algebra $\cH$ was first constructed as the Iwahori Hecke algebra for a $p$-adic Kac-Moody group (by Braverman, Kazhdan, and Patnaik in the untwisted affine case \cite{Braverman-Kazhdan-Patnaik}, and later in full generality by Bardy-Panse, Gaussent, and Rousseau \cite{BardyPanse-Gaussent-Rousseau-2016}). 

The algebra $\cH$ has a Bernstein presentation: $\cH$ is generated by a copy of the Hecke algebra $\cH_W$ of the Weyl group $W$ of the Kac-Moody root system and the semi-group algebra of the Tits cone $\cT$ inside the coweight lattice. In the case of affine root systems, the Bernstein presentation is a slight variation of the usual presentation of the double affine Hecke algebra \cite[\S 1.2.4]{Braverman-Kazhdan-Patnaik}.

Because of its $p$-adic realization, $\cH$ comes equipped with a $T$-basis $\{T_x\}_{x \in W \ltimes \Titscone}$. Here the indexing set is the semi-direct product $W \ltimes \cT$. This is not a Coxeter group, and is fact not even a group. Accordingly, there is no analogue of the Iwahori-Matsumoto presentation for $\cH$, and multiplying elements of the $T$-basis is a convoluted affair. We explain the precise algorithm in \S \ref{sssect:mult-in-T-basis}. Briefly, $T$-basis elements are expanded in Bernstein generators, the Bernstein generators are multiplied using the Bernstein relation, and the result is then expanded back into the $T$-basis. By contrast, multiplication of $T$-basis elements in the affine Hecke algebra $\cH_{\Waff}$ is straightforward: one simply uses the Coxeter presentation of $\Waff,$ without any need to involve the Bernstein basis. 

However, there is strong indication that a substitute for Coxeter theory should exist in this Kac-Moody affine context. For example, in recent years, we have found that $W \ltimes \cT$ carries many structures in common with Coxeter groups, such as a Bruhat order and a compatible length function \cite{Braverman-Kazhdan-Patnaik,Muthiah-2018,Muthiah-Orr-2019}. These constructions were inspired by Kac-Moody affine Hecke algebras, but until this present paper no precise connection was made between the two.

In this paper we aim to link these structures with $\cH$, and our results are the first steps toward the Coxeter theory of the Kac-Moody affine Hecke algebra. They connect the length function and the Bruhat order with $\cH$ and its $T$-basis. The results and conjectures are outlined in the remainder of this Introduction. For more precise statements about the context, see Section \ref{sect:prelim}, which contains a summary of the necessary background on Kac-Moody root systems and their affinization, the corresponding Weyl group $W$ and Tits cone $\cT$, the Bruhat order and length function on $W\ltimes \cT$, and the Kac-Moody affine Hecke algebra $\cH$ and its $T$-basis and Bernstein presentation.

\subsection{Length function and trivial/sign representation}

Our first result makes a precise connection between the length function on the semigroup $W\ltimes \cT$ and the Kac-Moody affine Hecke algebra $\cH$. Let $q$ denote the parameter for $\cH$ (see \S \ref{subsect:KMAHA}). In the realization of $\cH$ as the Iwahori-Hecke algebra of a $p$-adic Kac-Moody group, it is the order of the residue field. 
\begin{Theorem}[Theorems 
  \ref{thm:triv-representation} and 
  \ref{thm:sign-representation} in the main text]\label{introthm:repstrivialandsign}
  Let $\cH$ be a Kac-Moody affine Hecke algebra with $T$-basis indexed by $W \ltimes \cT$ where $W$ is the Weyl group and $\cT$ is the Tits cone in the lattice of coweights. Then there are ``trivial'' and ``sign'' representations 
  \begin{align}
    \triv:  \cH \rightarrow \ZZ[q,q^{-1}] \\
    \sign:  \cH \rightarrow \ZZ[q, q^{-1}]
  \end{align}
  characterized by
  \begin{align}
    \label{eq:9}
    \triv(T_x) = q^{\ell(x)} \\
    \sign(T_x) = (-1)^{\ell(x)} 
  \end{align}
 for all $x \in W \ltimes \cT$, where $\ell(x)$ is the length of $x$.
\end{Theorem}

The theorem and its proof form the contents of Section \ref{sect:trivial-and-sign}. There it is stated in another way: we first define these representations using Bernstein generators, and upon computing the images of the $T$-basis elements we {\it discover} the length function. We also mention that Theorem \ref{introthm:repstrivialandsign} above is precise only when the length function takes values in $\ZZ$ (which holds in finite and affine cases). To phrase the general statement precisely, we will need to allow denominators and correspondingly need to adjoint roots of $q$ and $-1.$

\subsection{Length deficits and products in the Hecke algebra}\label{subsect:intro-lengthdefandstructurecoeff}

The identity \eqref{eq:length-additive-intro} shows how products of elements in the $T$-basis behave for length-additive pairs $x,y$ in a Coxeter Weyl group. If $\ell(xy)<\ell(x)+\ell(y)$ (i.e. if $x,y$ are a length deficient pair) then the product is more complicated. To describe it, we need the following formula for length deficits.
\begin{Theorem}[Theorems  
    \ref{thm:finiteness-of-invx-intersect-invyinverse} and
    \ref{thm:length-deficit-for-tits-cone-weyl-groups} in the main text]
    \label{thm:intro-length-deficits}
  Fix a Kac-Moody root datum. Let $x,y \in W \ltimes \cT$. Then $\Inv(x) \cap \Inv(y^{-1})$ is a finite set, and \begin{equation}
    \label{eq:1}
    \ell(x) + \ell(y) = \ell(xy) + |\Inv(x) \cap \Inv(y^{-1})|
  \end{equation}
\end{Theorem}
Here $\Inv(x)$ and $\Inv(y^{-1})$ are inversions sets, i.e. the sets of positive Kac-Moody affine roots made negative by the corresponding element. This statement is well-known in the case of Coxeter Weyl groups with an easy proof. However, the proof in the Kac-Moody case is much more subtle since inversion sets are generally infinite. Section \ref{sect:KMaffine-lengtdef} is dedicated to proving first that the intersection on the right-hand side of \eqref{eq:1} is finite (Theorem \ref{thm:finiteness-of-invx-intersect-invyinverse}), then the identity \eqref{eq:1} (Theorem \ref{thm:length-deficit-for-tits-cone-weyl-groups}). 

Using this characterization of length deficits in terms of inversion sets, in section 4 we prove the following theorem for Coxeter Weyl groups.
\begin{Theorem}[Theorem \ref{thm:expansion-properties-Coxeter} in the main text]
  \label{thm:intro-support-of-product}
  
  Let $W$ be the Weyl group of a Kac-Moody root system, and let $\cH_W$ be the corresponding Hecke algebra. Let $x,y\in W$, and let $N = |\Inv(x) \cap \Inv(y^{-1})|$.
  When we expand 
	\begin{equation}
          \label{eq:expansion-of-Tx-Ty-intro}
	T_x\cdot T_y=\sum_{z \in W}c^{z}_{x,y}T_z
	\end{equation}
        in $\cH_W$ the following are true.
        \begin{enumerate}
        \item\label{part:boundedinBruhat-coxeterexpansionthm}  The coefficient $c^{z}_{x,y}$ is nonzero only if $z \geq xy$ in the Bruhat order.
        \item\label{part:lowerboundnumterms-coxeterexpansionthm} The coefficient $c^{xy}_{x,y}=q^{N}$, and $c^{xs_{\beta }y}_{x,y}\neq 0$ for all $\beta \in \Inv(x) \cap \Inv(y^{-1})$. In particular, we have an lower bound on the support of \eqref{eq:expansion-of-Tx-Ty-intro}: 
          \begin{equation}
            \label{eq:lower-bound-intro}
 \left|\left\lbrace z\in W\mid c^{z}_{x,y}\neq 0 \right\rbrace\right| \geq N + 1 
          \end{equation}
        \item\label{part:upperboundnumterms-coxeterexpansionthm} We also have an upper bound on the support of \eqref{eq:expansion-of-Tx-Ty-intro}:
	\begin{equation}\label{eq:nonzerocovercoeffs-intro}
	\left|\left\lbrace z\in W\mid c^{z}_{x,y}\neq 0 \right\rbrace\right|           \leq 2^{N}
	\end{equation} 
        \end{enumerate}
\end{Theorem}
Note two special cases. If $\ell(x) + \ell(y) = \ell(xy)$, then $N = 0$ by Theorem \ref{thm:intro-length-deficits}. By parts \eqref{part:lowerboundnumterms-coxeterexpansionthm} and \eqref{part:upperboundnumterms-coxeterexpansionthm} of the theorem, the right hand side of \eqref{eq:expansion-of-Tx-Ty-intro} is only one term with coefficient $1$. That is, we recover \eqref{eq:length-additive-intro}.

If $\ell(x) + \ell(y) = \ell(xy) + 2$, then $\Inv(x) \cap \Inv(y^{-1})$ is a singleton containing a single root $\beta$, and $N=1$.  In this case, the theorem says that $
T_x T_y = c T_{x s_\beta y} + q T_{xy}$
for some coefficient $c$. In the main text, we prove a more elaborated version of \Ref{thm:intro-length-deficits}, from which we compute this coefficient to be $q-1$. So in this case we have: 
\begin{equation}
  \label{eq:length-deficit-two-intro}
T_x T_y = (q-1) T_{x s_\beta y} + q T_{xy}  
\end{equation}
Note this case includes the special situation where $x$ and $y$ are both the same simple reflection, and then \eqref{eq:length-deficit-two-intro} is just the quadratic relation for the Hecke algebra \eqref{eq:T-identities}.

As our interest is Kac-Moody affine Hecke algebras, we state the following conjecture. 
\begin{Conjecture}[Conjecture \ref{conj:main-conjecture} in the main text]
  \label{conj:intro-main-conjecture}
  Theorem \ref{thm:intro-support-of-product} and formulas \eqref{eq:length-additive-intro} and \eqref{eq:length-deficit-two-intro} hold for Kac-Moody affine Hecke algebras.
\end{Conjecture}
Observe that Theorem \ref{thm:intro-length-deficits} is required to even state this conjecture. We check its validity for several cases for $\widehat{SL_2}$ in Section \ref{sec:conjecture-in-the-kac-moody-affine-setting}. 

Theorem \ref{thm:intro-support-of-product} and Conjecture \ref{conj:intro-main-conjecture} incorporate all the Coxeter-type phenomena that are currently known in the Kac-Moody affine case: the length function, the Bruhat order, inversion sets, and the $T$-basis of the Kac-Moody affine Hecke algebra. However, our proof of Theorem \ref{thm:intro-support-of-product} uses reduced words, which is an aspect of Coxeter combinatorics that does not generalise to the Kac-Moody affine case. So what makes Conjecture tantalising is that (1) it makes precise sense, (2) it is true in the case of Coxeter Weyl groups, (3) the proof in the Coxeter Weyl group case does not generalise, and yet (4) it appears to be true. Therefore, we expect proving the Conjecture will require uncovering a richer theory of Coxeter combinatorics for Kac-Moody affine Hecke algebras.

\subsection{Specializing $q=0$ and the Demazure product}

The $T$-basis of $\cH$ has structure coefficients lying in $\ZZ[q,q^{-1}]$ (see \S \ref{sssect:mult-in-T-basis}). However, more is true: the structure coefficients must be non-negative integers when $q$ is specialized to a prime power, and one deduces they in fact lie in $\ZZ[q]$ (\cite[\S 6.7]{BardyPanse-Gaussent-Rousseau-2016} and \cite[Theorem 3.22]{Muthiah-2018}). For affine Hecke algebras, this is an immediate consequence of the Iwahori-Matsumoto relations, but it is a much subtler statement in the Kac-Moody affine setting.
We may therefore consider the $q=0$ specialization of the basis. This is the content of Section \ref{sect:Demazure-product}. We state the following conjecture.

\begin{Conjecture}[Conjecture \ref{conj:q-equals-zero-specialization} in the main text]
Let $x,y \in W_{\cT}$. Then there exists a unique $x \star y \in W_{\cT}$ such that:
\begin{equation}
  \label{eq:d:22}
  T_x T_y \equiv (-1)^{\ell(x) + \ell(y)} T_{x \star y} \mod q
\end{equation}
\end{Conjecture}
In the case of affine Hecke algebras, we know that $\star$ is exactly the Demazure product \eqref{eq:d:2}. In the Kac-Moody affine context, we check \eqref{eq:d:22} in a number of cases for $\widehat{SL_2}$ in Section \ref{sec:conjecture-in-the-kac-moody-affine-setting}. Furthermore, the formulas \Ref{eq:length-additive-intro} and \Ref{eq:length-deficit-two-intro} give precise conjectures for $x\star y$ in the cases where the pair $x,y$ is length additive (Conjecture \ref{conj:demprod-lengthadd}) or has length-deficit-two (Conjecture \ref{conj:demprod-lengthdef2}).

Schremmer's recent work on the affine Demazure product \cite{Schremmer} offers another approach using quantum Bruhat graphs. Finally, we mention Lewis Dean has made progress in generalising Schremmer's approach and is developing the {\em Kac-Moody affine Demazure product}.

\subsection{Acknowledgements}
The first named author was supported by JSPS KAKENHI Grant Number JP19K14495. 
The second named author was supported by the grant DE200101802 of the Australian Research Council during some of the work on this paper. 
We thank Lewis Dean, August H\'ebert, and Manish Patnaik for stimulating conversations.

%%% Local Variables:
%%% mode: latex
%%% TeX-master: "main"
%%% End:

\section{Preliminaries}\label{sect:prelim}

\subsection{Kac-Moody root system}

Fix a Kac-Moody root datum indexed by a Dynkin diagram $I$. Let $\Delta $ denote the set of real roots of our Kac-Moody root system, and $\Delta^+\subset \Delta $ the set of positive real roots. Let $\{\alpha_i\}_{i\in I}$ be the set of simple roots.  We write $\beta >0$ for $\beta \in \Delta ^+$ and $\beta <0$ for $\beta \in \Delta-\Delta^+=-\Delta^+.$ For any $\beta \in \Delta$ we will write the same symbol $\beta$ for the corresponding real coroot. Whether the symbol refers to a real root or a real coroot will be clear from context. In particular, $\{\alpha_i\}_{i\in I}$ also denotes the set of simple coroots. Let $P$ denote the coweight lattice and $Q$ denote the coroot lattice.

Let $W$ denote the Weyl group, which is generated by the simple reflections $s_i$ for $i \in I$. Recall that $W$ is a Coxeter group, and we have the the length function $\ell:W\rightarrow \ZZ_{\geq 0}$ and Bruhat order $>$. Because $W$ is the Weyl group of a Kac-Moody root system, these structures can be interpreted using the root system. For $w\in W$ we have $\ell(w)=|\Inv(w)|,$ where $\Inv(w)=\Delta ^+\cap w^{-1}(-\Delta ^+)$. For $w\in W,$ $\beta \in \Delta^+ $, and $s_\beta $ the corresponding reflection we have 
\begin{align}\label{eq:def-Bruhat-classical}
 ws_{\beta}  > w &\text{ if } w(\beta )  \in \Delta^+  \\
 s_{\beta}w  >w &\text{ if } w^{-1}(\beta) \in \Delta^+  
\end{align}
and such inequalities generate the Bruhat order. Recall also that these generating inequalities can be phrased in terms of of the length function:
\begin{align}
  \label{eq:bruhat-ineq-length-ineq}
w(\beta) > 0 \text{ if and only if  } \ell(w s_{\beta}) > \ell(w)  
\end{align}
and:
\begin{align}
w^{-1}(\beta) > 0 \text{ if and only if }\ell(s_{\beta} w ) > \ell(w)
\end{align}

\subsection{Kac-Moody affine root system}

Following Braverman, Kazhdan, and Patnaik \cite[Appendix B]{Braverman-Kazhdan-Patnaik}, we define the set of Kac-Moody affine roots (KM affine roots for short) as :
\begin{equation}\label{eq:def:G-affine-roots}
\afD=\{\beta+n\pi \in \ZZ\Delta \oplus \ZZ\pi \mid \beta\in \Delta,\ n\in \ZZ \}
\end{equation}
and the set of positive KM affine roots as:
\begin{equation}\label{eq:def:pos-G-affine-roots}
\afDp=\{\beta+n\pi \in \afD \mid \beta\in \Delta^+,\ n\geq 0 \}\cup \{\beta+n\pi \in \afD \mid \beta\in -\Delta^+,\ n>0 \}
\end{equation}
Here $\pi$ is a formal symbol, and we write $\beta +n\pi >0$ if $\beta +n\pi \in \afDp$.

Following \cite[(5),(6)]{Muthiah-Orr-2019} we set 
\begin{equation}
\sgn:\ZZ\rightarrow \{\pm 1\};\ \sgn(n)=\left\lbrace\begin{array}{ll}
+1 & \text{ if }n\geq 0\\
-1 & \text{ if }n< 0\\
\end{array}\right.
\end{equation}
and for any $\beta\in \Delta ^+$ and $n\in \ZZ$ define $\beta[n]\in \afDp$
\begin{equation}
\beta[n]=\sgn(n)\cdot (\beta +n\pi )=\sgn(n)\beta +|n|\pi . 
\end{equation}
Note that $(\beta,n)\mapsto \beta[n]$ is a bijection $\Delta^+\times \ZZ\to \afDp.$ 

\subsection{The Tits cone and $\WTits$}

We denote elements of the semidirect product $W_P=W \ltimes P$ by $\pi^\mu w$ ($\mu\in P,$ $w\in W$). This group acts on the set $\afD$ via: 
\begin{equation}\label{eq:Weylaction-G-affine}
\pi^{\mu}w(\beta+n\pi )=w(\beta )+(n+\langle \mu, w(\beta )\rangle )\pi 
\end{equation}

For $\beta[n]\in \afDp$ we write $s_{\beta[n]}=\pi^{n\beta}s_{\beta }\in W_P$ for the corresponding reflection. Here the $\beta$ in the exponent of $\pi$ is the real coroot corresponding to $\beta \in \Delta$.
Analogously to the classical case, we may define 
\begin{equation}\label{eq:def:inv-doubleaffine}
\Inv(\pi^{\mu}w)=\{\beta[n]\in \afDp \mid \pi^{\mu}w(\beta[n])\notin \afDp\}.
\end{equation}
Note that {\emph{unlike}} in the classical setting, $\Inv(\pi^{\mu}w)$ is generally not a finite set. %(In fact unless $I$ is of finite type, it is infinite unless $\mu =0.$\anna{Is this true?})a \dinakar{This isn't true.} \anna{Oh.}

Let $P^+\subset P$ denote the set of dominant coweights (i.e. for $\lambda \in P$ we have $\lambda\in P^+$ if $\langle \alpha_i,\lambda  \rangle \geq 0$ for every $i\in I$). Let
\begin{equation}
  \label{eq:def-of-tits-cone}
  \Titscone = \left\{ \mu \in P \suchthat \mu = w(\lambda) \text{ for some } \lambda \in P^+ \text{ and } w\in W\right\} 
\end{equation}
denote the Tits cone, and let $\WTits=W\rtimes \Titscone$. This is a subsemigroup of $W_P$.

\subsection{The Bruhat order and length function on $\WTits$}
Fix a weight $\rho$ such that $\langle \rho, \alpha_i \rangle = 1$ for all $i \in I$. Such a $\rho$ always exists, and if the Cartan matrix is non-singular it is unique, but otherwise we have a choice. We choose $\rho$ such there exists $N \in \ZZ_{\geq 1}$ such that
\begin{equation}
\label{eq:d:20}
\langle \rho, \mu \rangle \in \ZZ \cdot \frac{1}{N}
\end{equation}
for all $\mu \in P$. We choose $\rho$ and $N$ so that $N$ is as small as possible.

\begin{Remark}
  In finite and affine cases, we will have $N=1$ or $N=2$, but for general Kac-Moody types arbitrarily large $N$ may appear. For example, for the rank-two hyperbolic type with Cartan matrix $\twomat{2}{-a}{-a}{2}$ with $a \geq 3$, we have $\langle \rho, \Lambda \rangle  = \frac{2}{4-a^2}$ for $\Lambda$ one of the two fundamental weights.
\end{Remark}

We define the {\em length function} $\ell : \WTits \rightarrow  \ZZ \cdot \frac{1}{N}$ by the following formula. For $\mu\in \Titscone$ and $w\in W$ we set (see e.g. \cite[(19)]{Muthiah-Orr-2019}):
\begin{equation}\label{eq:def-length-doubleaffine}
\ell(\pi^{\mu}w)=2\innprod{\rho}{\mu }+\sum_{\eta\in \Delta^+} \left\lbrace \begin{array}{ll}
-\innprod{2\eta}{\mu} & \text{ if }\innprod{\eta}{\mu}<0\text{ and }\eta\notin \Inv(w^{-1})\\
-\innprod{2\eta}{\mu}-1 & \text{ if }\innprod{\eta}{\mu}<0\text{ and }\eta\in \Inv(w^{-1})\\
1 & \text{ if }\innprod{\eta}{\mu}\geq 0\text{ and }\eta\in \Inv(w^{-1})\\
0 & \text{ if }\innprod{\eta}{\mu}\geq 0\text{ and }\eta\notin \Inv(w^{-1}).
\end{array}\right.
\end{equation}
Observe that, in particular, for $\mu \in \cT$, we have $\ell(\pi^{\mu})=2\innprod{\rho}{\mu _+}$ where $\mu_+ \in P^+$ is the unique dominant coweight in the Weyl orbit $W \cdot \mu$.

The {\em Bruhat order} $>$ on $\WTits$ is defined as follows. Suppose for $\pi^\mu w \in \WTits$, $\beta[n]\in \afDp$, and $\pi^{\mu} w s_{\beta[n]} \in \WTits$. Generalizing \eqref{eq:def-Bruhat-classical}, we have
\begin{align}\label{eq:def-Bruhat-doubleaffine}
\pi^{\mu}ws_{\beta[n]}  > \pi^{\mu} w &\text{ if }\pi^{\mu }w(\beta[n]) \in \afDp \\
s_{\beta[n]}\pi^{\mu}w  > \pi^{\mu}w \ &\text{ if } w^{-1}\pi^{-\mu }(\beta[n]) \in  \afDp
\end{align}
and the Bruhat order is generated by all such inequalities. 

This definition was first given by Braverman, Kazhdan, and Patnaik \cite{Braverman-Kazhdan-Patnaik}. Muthiah in \cite{Muthiah-2018} proved that the order is in fact a partial order, and Muthiah and Orr proved the following compatibility statement generalizing \eqref{eq:bruhat-ineq-length-ineq}:
\begin{align}
  \label{eq:km-bruhat-lenght-ineq}
\pi^\mu w(\beta[n]) \in \afDp  \text{ if and only if  } \ell(\pi^\mu w s_{\beta[n]}) > \ell(\pi^\mu w)  
\end{align}
 Following \cite[(15)]{Muthiah-Orr-2019} define:
\begin{equation}\label{eq:def:Inv++}
\Inv^{++}_{\pi^\mu w}(s_{\beta [n]}) = \left\lbrace \gamma [m]\in \Inv(s_{\beta [n]})\mid \pi^\mu w(\gamma [m])>0\text{ and } \pi^\mu w(-s_{\beta [n]}(\gamma [m]))>0 \right \rbrace 
\end{equation}
Muthiah and Orr prove the following more elaborated statement from which \eqref{eq:km-bruhat-lenght-ineq} can be deduced.
\begin{Theorem}[\mbox{\cite[Theorem 3.4]{Muthiah-Orr-2019}}]
Suppose $\pi^\mu w \in \WTits$, $\beta[n]\in \afDp$, $\pi^\mu w s_{\beta[n]} \in \WTits$, and $\pi^\mu w (\beta[n]) > 0$. Then $ \Inv^{++}_u(s_{\beta [n]})$ is a finite set and:
\begin{equation}
  \label{eq:2}
  \ell(\pi^\mu w s_{\beta[n]}) = \ell(\pi^\mu w) + | \Inv^{++}_{\pi^\mu w}(s_{\beta [n]})|
\end{equation}
\end{Theorem}

\subsection{The Kac-Moody affine Hecke algebra}\label{subsect:KMAHA}

Let $q$ be an indeterminate. We define the {\em Kac-Moody affine Hecke algebra} $\cH$ to be the $\ZZ[q^{1/N},q^{-1/N}]$ algebra generated by $\left\{ T_i \right\}_{i \in I}$ and $\left\{ \Theta^\mu \suchthat \mu \in \cT \right\}$ subject to the following relations. First, the $T_i$ generate a copy of $\cH_W$, the Hecke algebra for the Coxeter group $W$, i.e., for all $i \in I$ we have the quadratic relation
\begin{equation}
\label{eq:T-identities}
  T_i^2  = (q-1) T_i + q 
\end{equation}
and the $T_i$ satisfy the braid relations for the Coxeter group $W$. Because of the braid relations, for all $w \in W$, we can unambiguously define
\begin{equation}
  \label{eq:14}
  T_w = T_{i_1} \cdots T_{i_\ell} 
\end{equation}
where $w = s_{i_1} \cdots s_{i_\ell}$ is a reduced decomposition.
The generators $\Theta^\mu$ generate a copy of the semi-group algebra of $\cT$, i.e. for all $\mu_1 , \mu_2 \in \cT$, we have $\Theta^{\mu_1} \Theta^{\mu_2} = \Theta^{\mu_1 + \mu_2}$. Finally, we have the Bernstein relation
\begin{align}
  \label{eq:Bernstein}
  T_i \Theta^\mu - \Theta^{s_i(\mu)} T_i = (q - 1) \frac{\Theta^\mu - \Theta^{s_i(\mu)}}{1 - \Theta^{-\alpha_i}}
\end{align}
for all $i \in I$ and $\mu \in \cT$.

Observe that in the special case where our Kac-Moody root datum is of finite type, we recover the Bernstein presentation of the corresponding affine Hecke algebra.

\subsubsection{Iwahori-Hecke algebras of $p$-adic Kac-Moody groups and the $T$-basis}\label{subsect:prelim-IHalgebras}

Let $\cK$ be a non-archimedean local field (e.g. the $p$-adic numbers), and let $q$ be the size of the residue field. Given the fixed Kac-Moody root datum and $\cK$, one has the corresponding $p$-adic Kac-Moody group $G$.

The group $G$ does not satisfy the Cartan decomposition unless it is of finite type, but it has a sub-semigroup $G_+$ that does. Let $J$ denote the Iwahori-subgroup of $G$. The $J$ double cosets of $G_+$ are indexed by $\WTits$. For each $x \in \WTits$, we write $T_x$ for the indicator function of the corresponding double coset. The {\em Iwahori-Hecke algebra of $G$} is the set of $J$-biinvariant functions on $G_+$ supported on finitely many double cosets. This set has an algebra structure under convolution, and it is isomorphic to the Kac-Moody affine Hecke algebra $\cH$. Therefore, the set $\left\{ T_x \suchthat x \in \WTits \right\}$ forms a basis of $\cH.$ We will call this the \emph{$T$-basis} of $\cH$. These results were first proved by Braverman, Kazhdan, and Patnaik \cite{Braverman-Kazhdan-Patnaik} in untwisted affine types and later by Bardy-Panse, Gaussent, and Rousseau \cite{BardyPanse-Gaussent-Rousseau-2016} in full generality.

\subsubsection{Computing the $T$-basis algebraically}
We will only work with the $T$-basis algebraically. We explain how to express $T$-basis elements in the Bernstein generators using \cite{Muthiah-2018, BardyPanse-Gaussent-Rousseau-2016}. A word of warning: our conventions match \cite{Braverman-Kazhdan-Patnaik} and \cite{Muthiah-2018} but differ by signs from \cite{BardyPanse-Gaussent-Rousseau-2016}.

We can write any $x\in \WTits$ as $x = \pi^{w(\lambda)} v$ for $\lambda\in P^+$ dominant and $w, v \in W$. We have (\cite[\S 6.2.1]{Braverman-Kazhdan-Patnaik} and \cite[\S 5.7]{BardyPanse-Gaussent-Rousseau-2016}):
\begin{equation}
  \label{eq:d:1}
  T_{\pi^\lambda} = q^{\langle \rho, \lambda \rangle} \Theta^\lambda
\end{equation}
The motivation for defining $\cH$ over $\ZZ[q^{\pm1/N}]$ is the fact that $\langle \rho, \lambda \rangle \in \frac{1}{N} \ZZ$.

To find the basis element for $\pi^{w(\lambda)}$ we use the following formula (\cite[Corollary 3.19]{Muthiah-2018} and \cite[Corollary 4.3]{BardyPanse-Gaussent-Rousseau-2016}):
\begin{equation}
  \label{eq:d:2}
 T_{\pi^{w(\lambda)}} = T_{w^{-1}}^{-1} T_{\pi^\lambda} T_{w^{-1}} = q^{\langle \rho, \lambda \rangle}T_{w^{-1}}^{-1} \Theta^{\lambda} T_{w^{-1}}
\end{equation}

We will also need the following generalization of the Iwahori-Matusumoto relations (\cite[Theorem 3.1]{Muthiah-2018} and \cite[Proposition 4.1]{BardyPanse-Gaussent-Rousseau-2016}). For any $y\in \WTits$ and $i \in I$, we have:
\begin{align}\label{eq:Ts_onright}
  T_{ys_i}=
  \begin{cases}
T_{y}T_i & \text{if }x(\alpha_i)\in \afDp\\
T_{y}T_i^{-1} & \text{if }x(\alpha_i)\notin \afDp
  \end{cases}
\end{align}
We also record the left side version of this relation:
\begin{align}
 \label{eq:Ts_onleft} 
  T_{s_i y} =
  \begin{cases}
    T_i T_y & \text{if }y^{-1}(\alpha_i) \in \afDp \\
    T_i^{-1}T_y & \text{if }y^{-1}(\alpha) \notin \afDp
  \end{cases}
\end{align}

We are now ready express $T_x$ for $x=\pi^{w(\lambda)}v$ as above. We choose a decomposition $v = s_{i_1} \cdots s_{i_r}$ into simple reflections. By repeated application of \eqref{eq:Ts_onright}, we compute
\begin{equation}
  \label{eq:Tbasis-in-Bernstein-generators}
  T_x =  T_{\pi^{w(\lambda)}v} = T_{\pi^{w(\lambda)}} T_{i_1}^{\eps_1} \cdots T_{i_r}^{\eps_r} = q^{\langle \rho, \lambda \rangle}T_{w^{-1}}^{-1} \Theta^{\lambda} T_{w^{-1}} T_{i_1}^{\eps_1} \cdots T_{i_r}^{\eps_r}
\end{equation}
where: 
\begin{align}
  \label{eq:signs-epsilon-k}
  \eps_k =
  \begin{cases}
1 & \text{if }\pi^{w(\lambda)} s_{i_1} \cdots s_{i_{k-1}}( \alpha_{i_{k}} ) \in \afDp\\
-1 & \text{if }\pi^{w(\lambda)} s_{i_1} \cdots s_{i_{k-1}}( \alpha_{i_{k}} ) \notin \afDp
  \end{cases}
\end{align}

\subsubsection{The Braverman-Kazhdan-Patnaik algorithm}\label{sssect:BKPalgorithm}

Formula \eqref{eq:Tbasis-in-Bernstein-generators} expresses $T$-basis elements in terms of Bernstein generators. Now we will explain an algorithm, due to Braverman, Kazhdan, and Patnaik, for expressing Bernstein generators $\Theta^{\mu},$ $\mu\in \cT$ in terms of the $T$-basis.

For arbitrary $\mu \in \Titscone$, let $\mu_+$ be its dominant Weyl-conjugate ($\mu_+\in W\mu\cap P^+$). Observe that $\mu \leq \mu_+$ in the dominance order. Consider the set:
\begin{align}
  \label{eq:3}
  \cS(\mu) = \left\{ \nu \suchthat \mu \leq \nu  \leq \mu_+ \textand \nu_+ \leq \mu_+ \right \}
\end{align}
It is clear that $\cS(\mu)$ is a finite set, and we will write down a formula for $\Theta^\mu$ by induction on the size of $\cS(\mu)$. Observe that $\{\mu,\mu_+\}\subseteq \cS(\mu)$ and $\#\cS(\mu) = 1$ if and only if $\mu$ is dominant. In this case, we have \eqref{eq:d:1}, which gives the base case of our recursion.

Now let us consider $\mu$ that is not dominant, and by induction assume that we have a formula for $\Theta^\kappa$ in the $T$-basis for all $\kappa$ with $\# \cS(\kappa) < \# \cS(\mu)$. Because $\mu$ is not dominant, there is a simple root $\alpha_i$ such that $\langle \alpha_i, \mu \rangle < 0$. Let $L = - \langle \alpha_i, \mu \rangle$. By the Bernstein relation:
\begin{align}
  \label{eq:Bernstein}
  T_i \Theta^{s_i(\mu)} - \Theta^{\mu} T_i = (q-1) ( \Theta^{s_i(\mu)} + \cdots + \Theta^{s_i(\mu) - (L-1) \alpha_i} )
\end{align}
which we can rewrite as
\begin{align}
  \label{eq:6}
  \Theta^{\mu} =\left( T_i \Theta^{s_i(\mu)} - (q-1) ( \Theta^{s_i(\mu)} + \cdots + \Theta^{s_i(\mu) - (L-1) \alpha_i} ) \right) T_i^{-1}
\end{align}
Consider the terms $\Theta^{\kappa}$ ($\kappa\in P$) appearing on the right hand side of \eqref{eq:6}. It is clear that $\mu\lneq \kappa $ in the dominance order. Furthermore, it is also clear that each such $\kappa $ is a weight appearing in the highest weight integrable module $V(\mu_+)$. So $\kappa_+ \leq \mu_+$ (where $\kappa_+\in W\kappa\cap P^+$). It follows that
\begin{align}
  \label{eq:7}
    \cS(\kappa) \subsetneq \cS(\mu)
\end{align}
for all $\kappa\in P$ with $\Theta^\kappa$ appearing on the right hand side of \eqref{eq:6}. By induction, \eqref{eq:6} together with the identities \eqref{eq:Ts_onright} and \eqref{eq:Ts_onleft} give a formula for expanding $\Theta^{\mu}$ into the $T$-basis.

\subsubsection{Multiplication in the $T$-basis}\label{sssect:mult-in-T-basis} 
We are interested in the structure coefficients of the $T$-basis. By the construction of $\cH$ as an Iwahori-Hecke algebra of a Kac-Moody group over a local field outlined in \ref{subsect:prelim-IHalgebras},  we know that they are equal to coefficients in a convolution when $q$ is specialized to a prime power. However, we will only work algebraically and compute structure coefficients by (1) expanding $T$-basis elements into the Bernstein generators using \eqref{eq:Tbasis-in-Bernstein-generators}, (2) multiplying the resulting expressions using the Bernstein relations and Hecke algebra relations, and (3) expanding the resulting expressions in the $T$-basis using the Braverman-Kazhdan-Patnaik algorithm. This algorithm shows that the structure coefficients lie in $\ZZ[q,q^{-1}]$ (which was not {\em a priori} clear from the convolution perspective), see \cite[\S 6.7]{BardyPanse-Gaussent-Rousseau-2016} and \cite[Theorem 3.22]{Muthiah-2018}. This algorithm is carried out for a few examples in Section \S \ref{sec:conjecture-in-the-kac-moody-affine-setting}.

\section{The trivial and sign representations}     \label{sect:trivial-and-sign}   
In this section, we construct the trivial and sign representations of $\cH$, generalizing the corresponding notions for affine Hecke algebras. As a consequence, we establish the first precise link between the length function and $\cH$.

\subsection{The trivial representation}

\begin{Theorem}
  \label{thm:triv-representation}
  There is an algebra map
  \begin{align}
    \label{eq:d:21}
    \triv:  \cH \rightarrow \ZZ[q^{1/N},q^{-1/N}]
  \end{align}
  defined on generators by
  \begin{align}
    \label{eq:d:22a}
    T_i &\mapsto q \\
    \label{eq:d:22b}
    \Theta^\mu &\mapsto q^{\langle \rho, \mu \rangle}
  \end{align}
  For $x \in \WTits$, we have
  \begin{align}
    \label{eq:d:23}
   T_x \mapsto q^{\ell(x)}
  \end{align}
  under $\triv$.
\end{Theorem}

Following the usual terminology, we will call $\triv$ the \emph{trivial representation} of $\cH$.

\begin{proof}

Observe that formulas \eqref{eq:d:22a} and \eqref{eq:d:22b} are compatible with the defining relations of $\cH$, which proves the existence of $\triv$.

Let $x \in \WTits$, which we write as $x = \pi^{w(\lambda)} v$ for $v,w \in W$ and $\lambda \in P^+$. Then we have by \eqref{eq:Tbasis-in-Bernstein-generators}
\begin{align}
  \label{eq:17}
  \triv(T_x) = q^{2 \langle \rho, \lambda \rangle + \eps_1 + \cdots + \eps_r}
\end{align}
where the $\eps_1, \ldots, \eps_r$ are defined by \eqref{eq:signs-epsilon-k}. On the other hand,  we have $\ell(\pi^{w(\lambda)}) = \ell(\pi^\lambda) = 2 \langle \rho, \lambda \rangle$, and by repeated application of \cite[Lemma 4.18]{Muthiah-2018} $\ell(x) = \ell(\pi^{w(\lambda)}) + \eps_1 + \cdots + \eps_r$.
\end{proof}

\begin{Remark}
  In \cite{Muthiah-2018}, we defined the length function to have properties mimicking the Iwahori-Matsumoto relations of $\cH$, but beyond analogy we did not make any statement of a precise relation between the two. In retrospect, we see that the precise relationship between the length function and $\cH$ is given by Theorem \ref{thm:triv-representation}.

  In particular, we can use Theorem \ref{thm:triv-representation} as a definition of the length function. That is, we define the trivial representation in the Bernstein generators, and prove that each T-basis element maps to a power of $q$. We then define the length to be exactly the power of $q$. The original definition of the length function is then a consequence of the Iwahori-Matsumoto relations. 
\end{Remark}

\subsection{The sign representation}

\begin{Theorem}
  \label{thm:sign-representation}
  There is an algebra map
  \begin{align}
    \label{eq:d:21-sign}
    \sign:  \cH \rightarrow : \ZZ[ e^{  \pi i \frac{1}{N}}][q^{1/N},q^{-1/N}] 
  \end{align}
  defined on generators by
  \begin{align}
    \label{eq:d:22a-sign}
    T_i \mapsto e^{\pi i } = -1 \\
    \label{eq:d:22b-sign}
    \Theta^\mu \mapsto e^{ \pi i {\langle \rho, \mu \rangle}}
  \end{align}
  For $x \in \WTits$, we have
  \begin{align}
    \label{eq:d:23-sign}
   T_x \mapsto e^{ \pi i {\ell(x)}}
  \end{align}
  under $\sign$.
\end{Theorem}

The proof of this theorem is a slight variation of the proof of Theorem \ref{thm:triv-representation}.

%%% Local Variables:
%%% mode: latex
%%% TeX-master: "main"
%%% End:

\section{Results for Weyl groups}\label{sect:Coxeter-results}

Recall that $W$ is the Weyl group of our Kac-Moody root system. In particular, $W$ is a Coxeter group, and $\cH_W$ is the Hecke algebra of this Coxeter group. It has basis $\{T_w\mid w\in W\}.$ The main result of this section is Theorem \ref{thm:expansion-properties-Coxeter} about structure coefficients of this basis. As outlined in \S \ref{subsect:intro-lengthdefandstructurecoeff} of the Introduction, the statement of this result makes sense in the Kac-Moody affine context, and is the subject of Conjecture \ref{conj:main-conjecture}. A crucial ingredient is the notion of length deficit, which is explained in Proposition \ref{prop:length-deficit-for-coxeter-weyl-groups}. It will be revisited in the Kac-Moody affine setting in Theorem \ref{thm:length-deficit-for-tits-cone-weyl-groups}. 

Given $x,y \in W$, we call the quantity $\ell(x) + \ell(y) - \ell(xy)$ the \emph{length deficit} of the pair $x$ and $y$. The following well-known result relates length deficits to inversion sets.

\begin{Proposition}
  \label{prop:length-deficit-for-coxeter-weyl-groups}
For $x,y \in W$, we have:
  \begin{equation}
    \label{eq:length-deficits-in-coxeter-case}
    \ell(x) + \ell(y) = \ell(xy) + 2 | \Inv(x) \cap \Inv(y^{-1}) |
  \end{equation}
\end{Proposition}

\begin{proof}
  This is equivalent to $| \Inv(x)| + |\Inv(y^{-1})|=|\Inv((xy)^{-1})| + 2 | \Inv(x) \cap \Inv(y^{-1}) |$. It therefore suffices to construct a bijection
  \begin{equation}
    \label{eq:d:2}
 F:  \Inv(x) \backslash \Inv(y^{-1})  \sqcup \Inv(y^{-1}) \backslash \Inv(x) \rightarrow \Inv((xy)^{-1}).
  \end{equation}
For $\beta \in \Inv(x) \setminus \Inv(y^{-1})$ set $F(\beta) = - x(\beta)$. Observe that this is a positive root because $\beta \in \Inv(x)$, and also that $(xy)^{-1}( F(\beta)) = - y^{-1}(\beta)$, which is a negative root because $\beta \notin \Inv(y^{-1})$. For $\gamma \in \Inv(y^{-1}) \backslash \Inv(x)$ set $F(\gamma) = x(\gamma)$. This is an element of $\Inv((xy)^{-1})$ by an argument analogous to the above.
  
To prove $F$ is a bijection, let us construct the inverse map. Let $\zeta \in \Inv((xy)^{-1})$. If $x^{-1}(\zeta) < 0$, then define $F^{-1}(\zeta) = -x^{-1}(\zeta) \in \Inv(x) \backslash \Inv(y^{-1})$. If $x^{-1}(\zeta) > 0$, then we define $F^{-1}(\zeta) = x^{-1}(\zeta) \in \Inv(y^{-1}) \backslash \Inv(x)$.
\end{proof}

We now state the main theorem of this section.

\begin{Theorem}\label{thm:expansion-properties-Coxeter}
  Let $x,y\in W$, and let $N = |\Inv(x) \cap \Inv(y^{-1})|$. Consider the expansion 
	\begin{equation}
          \label{eq:expansion-of-Tx-Ty}
	T_x\cdot T_y=\sum_{z \in W}c^{z}_{x,y}T_z
	\end{equation}
        in $\cH_W.$ The following are true.
        \begin{enumerate}
        \item\label{thpart:boundedinBruhat-coxeterexpansionthm}  The coefficient $c^{z}_{x,y}$ is nonzero only if $z \geq xy$ in the Bruhat order.
        \item\label{thpart:lowerboundnumterms-coxeterexpansionthm} The coefficient $c^{xy}_{x,y}=q^{N}$, and $c^{xs_{\beta }y}_{x,y}\neq 0$ for all $\beta \in \Inv(x) \cap \Inv(y^{-1})$. In particular, we have an lower bound on the support of \eqref{eq:expansion-of-Tx-Ty}: 
          \begin{equation}
            \label{eq:lower-bound}
 \left|\left\lbrace z\in W\mid c^{z}_{x,y}\neq 0 \right\rbrace\right| \geq N + 1 
          \end{equation}
        \item \label{thpart:upperboundnumterms-coxeterexpansionthm}We also have an upper bound on the support of \eqref{eq:expansion-of-Tx-Ty}:
	\begin{equation}\label{eq:nonzerocovercoeffs}
	\left|\left\lbrace z\in W\mid c^{z}_{x,y}\neq 0 \right\rbrace\right|           \leq 2^{N}
	\end{equation} 
        \item\label{thpart:degreeboundoncoeffs} For all $z$, $c^z_{x,y}\in \ZZ[q]$. If $c^z_{x,y}\neq 0$ then $\deg c^z_{x,y}\leq N$ with positive leading coefficient, and the lowest degree term of $c^z_{x,y}$ is degree at least $\ell(xy)+N-\ell(z)$ and has sign $(-1)^{\ell(z)-\ell(xy)}$.
        \end{enumerate}
\end{Theorem}

\begin{Remark}\label{rmk:Coxeter-demazure-product}
  We were not able to find in the literature the Bruhat lower bound (Theorem \Ref{thm:expansion-properties-Coxeter}(i)), but there is a well-known Bruhat upper bound. Let $x\ast y$ denote the Demazure product of $x$ and $y$ (see e.g. \cite{kenney2014coxeter}). Then the coefficient $c^{z}_{x,y}$ in \eqref{eq:expansion-of-Tx-Ty} is nonzero only if $z \leq x \star y$ in the Bruhat order. Moreover, the coefficient $c^{x \star y}_{x,y}$ is non-zero. In the specialisation $q=0,$ $c^{x \star y}_{x,y}$ becomes 
  $(-1)^{\ell(x)+\ell(y)}$ and all other terms vanish \cite[p. 721]{kenney2014coxeter}. By contrast, the $q=1$ specialisation of $\cH_W$ is ${\mathbb{Z}} W$ and the only surviving term is $c^{xy}_{x,y}T_{xy}=1\cdot xy.$ We shall revisit this in Section \S\ref{sect:Demazure-product}. 
\end{Remark}

For $x,y\in W$, we introduce the following abbreviated notation:
\begin{equation}\label{eq:def:curlyN}
\cN(x,y)=\Inv(x)\cap\Inv(y^{-1})
\end{equation}

In preparation for the proof, we record a few lemmas. Lemma \ref{lem:ellisN} serves to elucidate the presence of the quantity $|\cN(x,y)|$ in Theorem \ref{thm:expansion-properties-Coxeter} above. Lemmas \ref{lem:upstep-auxiliary} and \ref{lem:downstep-auxiliary} collect simple statements about quantities in the Theorem change when $x$ or $y$ are multiplied by a simple reflection. 

\begin{Lemma}\label{lem:ellisN}
	Let $x,y\in W$ with a reduced decomposition $y=s_{j_1}\cdots s_{j_n}.$ Set $x=x_0,$ $x_i=x_{i-1}s_{j_i}$ for $1\leq i\leq n.$ Then
	\begin{equation}\label{eq:def:el}
	|\{1\leq i\leq n\mid  x_{i-1}>x_i\}|=|\cN(x,y)|.
	\end{equation}
\end{Lemma}
\begin{proof} For $1\leq i\leq n$ we have $\ell(x_i)-\ell(x_{i-1})=\pm 1,$ hence:
	\begin{equation*}
\ell(xy)-\ell(x)=\sum_{i=1}^{n}(\ell(x_i)-\ell(x_{i-1})) =  \ell(y)-2|\{1\leq i\leq n\mid  x_{i-1}>x_i\}|\\
	\end{equation*}	
	The statement then follows from Proposition \ref{prop:length-deficit-for-coxeter-weyl-groups}.
\end{proof}

\begin{Lemma}\label{lem:upstep-auxiliary}
	Let $s=s_{\alpha }\in W$ be a simple reflection and $x,y\in W$ such that $x'=xs,$ $x<x'$ and $y=sy',$ $y'<y.$ Then we have:
	\begin{enumerate}[(a)]
		\item\label{lempart:upstepaux-a} $x'y'=xy$
		\item\label{lempart:upstepaux-b} $\cN(x',y')=s(\cN(x,y))$
		\item\label{lempart:upstepaux-c} For any $\beta$ root one has $xs_{\beta }y=x'ss_{\beta }sy'=x's_{s(\beta )}y'.$
		\item\label{lempart:upstepaux-d} $|\cN(x',y')|=|\cN(x,y)|$ and $\ell(x'y')+|\cN(x',y')|=\ell(xy)+|\cN(x,y)|.$
	\end{enumerate}
\end{Lemma}
\begin{proof}
	The statements in \eqref{lempart:upstepaux-a} and \eqref{lempart:upstepaux-c} are immediate. The conditions imply that $\Inv(x')=s(\Inv(x)\setminus \{\alpha \})$ and  $\Inv(y'^{-1})=s(\Inv(y^{-1})\setminus\{\alpha \})$ (cf. \eqref{eq:def-Bruhat-classical}) whence \eqref{lempart:upstepaux-b} follows. Then \eqref{lempart:upstepaux-d} follows from \eqref{lempart:upstepaux-a} and \eqref{lempart:upstepaux-b}.
\end{proof}

\begin{Lemma}\label{lem:downstep-auxiliary}
	Let $s=s_{\alpha }$ be a simple reflection and $x,y\in W$ such that $x'=xs,$ $x>x'$ and $y=sy',$ $y'<y.$ Then we have:
	\begin{enumerate}[(a)]
		\item\label{lempart:downstepaux-a} $xy=x'y'<xy'$ and $2\nmid \ell(xy')-\ell(x'y')$
		\item\label{lempart:downstepaux-b} $\cN(x,y)=s(\cN(x',y'))\cup \{\alpha\}$ 
		\item\label{lempart:downstepaux-c} $xy'=xs_{\alpha}y,$ for $\beta \in \cN(x,y)\setminus \{\alpha \}$ one has $xs_{\beta }y=x'ss_{\beta }sy'=x's_{s(\beta )}y'$
		\item\label{lempart:downstepaux-d} $\ell(xy)+|\cN(x,y)|=\ell(x'y')+|\cN(x',y')|+1\leq \ell(xy')+|\cN(x,y')|$
		\item\label{lempart:downstepaux-e} $|\cN(x,y')|\leq |\cN(x,y)|-1=|\cN(x',y')|$
	\end{enumerate}
\end{Lemma}
\begin{proof}
	In \eqref{lempart:downstepaux-a} $xy=x'y'$ is clear, and we have $x'y'=xy'\cdot y'^{-1}sy'=xy'\cdot s_{y'^{-1}\alpha }$. The statement then follows from \eqref{eq:def-Bruhat-classical}, since $y'^{-1}\alpha >0$ and $xy'(y'^{-1}\alpha )=x(\alpha )<0.$
Part \eqref{lempart:downstepaux-b} is true because the conditions imply that $\Inv(x')=s(\Inv(x)\setminus\{\alpha \})$ and $\Inv(y'^{-1})=s(\Inv(y^{-1})\setminus\{\alpha \}).$ 
Part \eqref{lempart:downstepaux-c} is clear. 
In part \eqref{lempart:downstepaux-d}, the equality follows from \eqref{lempart:downstepaux-b}. The inequality follows from
	\begin{equation*}
	\begin{split}
	\ell(xy')+|\cN(x,y')|=&\frac{1}{2}(\ell(x)+\ell(y')+\ell(xy'))\\%=\frac{1}{2}(\ell(x')+1+\ell(y')+\ell(x'y')+\ell(xy')-\ell(x'y'))\\
	=&\ell(x'y')+|\cN(x',y')|+\frac{1+\ell(xy')-\ell(x'y')}{2}\\
	\geq & \ell(x'y')+|\cN(x',y')|+\frac{1+1}{2}
	\end{split}
	\end{equation*}
	using Proposition \ref{prop:length-deficit-for-coxeter-weyl-groups} and \eqref{lempart:downstepaux-a}.
	Part (e) follows from part (b) and
	\begin{equation*}
\begin{split}
2|\cN(x,y')|=&\ell(x)+\ell(y')-\ell(xy')\\
=&(\ell(x,y)+\ell(y)-\ell(xy))-2+(\ell(xy)+1-\ell(xy'))\\
\leq &2|\cN(x,y)|-2
\end{split}
\end{equation*}
using Proposition \ref{prop:length-deficit-for-coxeter-weyl-groups} and part (a).
\end{proof}

\subsubsection{Proof of Theorem \ref{thm:expansion-properties-Coxeter}}
The proof is by induction on the length of $y.$ The statement is clear for $y=1_W.$ If $\ell(y)=1,$ i.e. $y=s=s_{\alpha }$ is a simple reflection, then we have: 
\begin{equation}\label{eq:basecase}
T_xT_s=\left\lbrace \begin{array}{ll}
T_{xs}\cdot T_s^2=T_{xs}\cdot ((q-1)T_s+q)=q^1\cdot T_{xy}+(q-1)T_{xs_{\alpha}y} & \text{if }x>xs\\
T_{xs}=q^0\cdot T_{xy}& \text{if }x<xs
\end{array}\right.
\end{equation}
Observe that if $x>xs$ then $\cN(x,y)=\{\alpha \},$ $N=1.$ If $x<xs$ then $\cN(x,s)=\emptyset,$ $N=0.$ In both cases, all statements of the theorem are true for the pair $(x,y).$

Now let us assume that the statement is true with the factor on the right having length at most $\ell(y)-1.$ Let $y=sy'$ with $\ell(y')=\ell(y)-1$ and $s=s_{\alpha }$ a simple reflection. Thus $T_y=T_sT_{y'}.$ Consider any $x.$ We have either $\ell(xs)>x$ or $\ell(xs)<x.$

If $\ell(xs)>x$ then we have $T_xT_y=T_xT_sT_{y'}=T_{xs}T_{y'}.$ Setting $x'=xs>x$ and $y'=sy<y$ we have $c^z_{x,y}=c^z_{x',y'}.$ The statements of the theorem then all follow from the inductive hypothesis for the pair $(x',y')$ and the statements of Lemma \ref{lem:upstep-auxiliary}.

If $\ell(xs)<x$ then let us write $x'=xs<x$ and $y'=sy<y,$ i.e. $x,y$ satisfy the conditions of Lemma \ref{lem:downstep-auxiliary}. We have 
\begin{equation}
T_xT_y=T_{x'}T_s^2T_{y'}=T_{x'}((q-1)T_s+q)T_{y'}=(q-1)T_{x}T_{y'}+qT_{x'}T_{y'}
\end{equation}
and furthermore by the inductive hypothesis:
\begin{equation}\label{eq:downto_y'products}
T_xT_y=(q-1)\sum_{z\geq xy'}c^{z}_{x,y'}T_z +q\sum_{z\geq x'y'}c^{z}_{x',y'}T_z
\end{equation}
It follows that 
\begin{equation}\label{eq:recursivecoeff-downstep}
c^z_{x,y}=(q-1)\cdot c^z_{x,y'}+q\cdot c^z_{x',y'}.
\end{equation}
We now check that the statements of the theorem are satisfied by the inductive hypothesis on the pairs $(x,y')$ and $(x',y')$ and parts of Lemma \ref{lem:downstep-auxiliary}. 

Indeed part \eqref{lempart:downstepaux-a} and $xy=x'y'<xy'$ implies that $c^z_{x,y}=0$ unless $z\geq xy.$ For $z=xy$ the only contribution comes from the right-hand sum, and so 
$$c^{xy}_{x,y}=q\cdot c^{x'y'}_{x',y'}=q^{1+|\cN(x',y')|}=q^{|\cN(x,y)|}.$$

By part \eqref{lempart:downstepaux-e} of the Lemma we have:
$$\left|\left\lbrace z\in W\mid c^{z}_{x,y}\neq 0 \right\rbrace\right|\leq 2^{|\cN(x,y')|}+2^{|\cN(x',y')|}\leq 2^{|\cN(x,y)|-1}+2^{|\cN(x,y)|-1}=2^{|\cN(x,y)|}$$

Observe that nonzero polynomials satisfying the conditions on $c^z_{x,y}$ (for any pair) have positive leading coefficient. In particular \eqref{eq:recursivecoeff-downstep} implies that $c^z_{x,y}$ is nonzero if at least one of $c^z_{x,y'}$ or $c^z_{x',y'}$ does not vanish. Let $\beta \in \cN(x,y).$ 
By part \eqref{lempart:downstepaux-b} of the Lemma either $\beta=\alpha $ or $s(\beta)\in \cN(x',y').$ Therefore by part \eqref{lempart:downstepaux-c} of the Lemma $c^{xs_\beta y}_{x,y}$ does not vanish: indeed $c^{xs_{\alpha}y}_{x,y'}\neq 0$ and $c^{xs_{\beta}y}_{x',y'}\neq 0$ for $s(\beta)\in \cN(x',y').$

Now let us examine the finer properties of the polynomials $c^z_{x,y}.$ The leading coefficient of both $c^z_{x,y'}$ and $c^z_{x',y'}$ is positive, and by part \eqref{lempart:downstepaux-e} of the Lemma $\deg c^z_{x,y'}\leq |\cN(x,y')|\leq |\cN(x,y)|-1,$ and $\deg c^z_{x',y'}\leq |\cN(x',y')|=|\cN(x,y)|-1.$ It follows that 
$$\deg c^z_{x,y}=1+\max (\deg c^z_{x,y'},\deg c^z_{x,y'})\leq |\cN(x,y)|$$
and the leading coefficient of $c^z_{x,y}$ is positive. 

By part \eqref{lempart:downstepaux-d} and \eqref{lempart:downstepaux-a} of the Lemma if $c^z_{x,y'}\neq 0$ then its lowest degree term has degree at least $\ell(xy')+|\cN(x,y')|-\ell(z)\geq \ell(xy)+|\cN(x,y)|-\ell(z)$ and has sign $(-1)^{\ell(z)-\ell(xy')}=-(-1)^{\ell(z)-\ell(xy)}.$ Similarly if $c^z_{x',y'}\neq 0$ then its lowest degree term has degree at least $\ell(x'y')+|\cN(x',y')|-\ell(z)= -1+\ell(xy)+|\cN(x,y)|-\ell(z)$ and has sign $(-1)^{\ell(z)-\ell(x'y')}=(-1)^{\ell(z)-\ell(xy)}.$ It follows from \eqref{eq:recursivecoeff-downstep} that if at least one of these coefficients is nonzero, then $c^z_{x,y}\neq 0$ and the lowest degree term of $c^z_{x,y}$ has degree at least $\ell(xy)+|\cN(x,y)|-\ell(z)$ and has sign $(-1)^{\ell(z)-\ell(xy)}.$
\qed

\subsection{Special case of length deficit 0, 2, 4}

Let $x,y \in W$. We consider the special cases where $x$ and $y$ have length deficits $0$, $2$, and $4$.

\begin{Example}\label{example:lengthdeficit-0-Coxeter}
  If $x$ and $y$ have length deficit $0$, then $\Inv(x)\cap \Inv(y)=\emptyset$. In this case, Theorem \ref{thm:expansion-properties-Coxeter} tells us that:
  \begin{align}
   T_xT_y=T_{xy} 
  \end{align}
\end{Example}

\begin{Example}\label{example:lengthdeficit-1-Coxeter}
 If $x$ and $y$ have length deficit $2$, then $\Inv(x)\cap \Inv(y)= \{\beta \}$ for some positive root $\beta$. In this case, Theorem \ref{thm:expansion-properties-Coxeter} tells us that $T_xT_y= c T_{x s_\beta y} +  q T_{xy}$ 
 for some coefficient $c \in \ZZ[q]$. We can compute the coefficient $c$ using Remark \ref{rmk:Coxeter-demazure-product} as follows. We have $xs_{\beta }y>xy,$ so the Demazure product is $x \star y = x s_\beta y.$ By part \eqref{thpart:degreeboundoncoeffs} of Theorem \ref{thm:expansion-properties-Coxeter}, we know that $c$ has degree at most $1$ in $q,$ its leading coefficient is positive and its lowest degree term in $q$ is negative. So $c = a q + b$ for some $a,b \in \ZZ$ with $a >  0$ and $b<0$. Since the value of $c$ is $0$ at $q=1$ and $\pm 1$ at $q=0,$ we get $a=-b=1,$ so $c=q-1.$ So in this case:
 \begin{align}
   \label{eq:length-deficit-two-T-basis-product}
  T_xT_y= (q-1) T_{x s_\beta y} +  q T_{xy}
 \end{align}
\end{Example}

\begin{Example}\label{example:lengthdeficit-2-Coxeter}
  If $x$ and $y$ have length deficit $4$, then Theorem \ref{thm:expansion-properties-Coxeter} implies that the support of $T_{x}T_y$ may have $3$ or $4$ elements. Both possibilities can in fact occur, as we illustrate below. Let $W$ be the Weyl group of $SL_3$ (the symmetric group on three letters), and $s_1$ and $s_2$ the simple reflections.

For a pair with a support of $3$ elements, take $x = s_1 s_2$ and $y = s_2 s_1.$ Note $xy=1_W.$ We compute:  
  \begin{equation}
    \label{eq:length-deficit-2-support-3}
    T_{s_1 s_2} T_{s_2 s_1} = (q-1) T_{s_1 s_2 s_1} + q(q - 1) T_{s_1} + q^2T_{1_W} 
  \end{equation}
 For a pair with a support of $4$ elements, take $x = s_1 s_2$ and $y = s_1 s_2 s_1 = s_2 s_1 s_2$. Note $xy=s_2.$ We compute:
\begin{equation}
  \label{eq:length-deficit-2-support-4}
 T_{s_1s_2} T_{s_1s_2s_1} = (q-1)^2 T_{s_1 s_2 s_1} + q(q-1) T_{s_2 s_1} q(q-1) T_{s_1 s_2} + q^2 T_{s_2}
\end{equation} 

\end{Example}

%%% Local Variables:
%%% mode: latex
%%% TeX-master: "main"
%%% End:

\section{Length deficits in the Kac-Moody affine setting}\label{sect:KMaffine-lengtdef}

The main result of the section is Theorem \ref{thm:length-deficit-for-tits-cone-weyl-groups}. This is the Kac-Moody affine generalization of Proposition \ref{prop:length-deficit-for-coxeter-weyl-groups}. The proof however is much more subtle, since the inversion sets of elements of $\WTits$ are usually infinite, and the relationship between lengths and inversion sets is more delicate.

\subsection{Finite intersections of inversions sets}

\begin{Theorem}
	\label{thm:finiteness-of-invx-intersect-invyinverse}
	Let $x,y \in W_\cT$. Then $\Inv(x) \cap \Inv(y^{-1})$ is a finite set.
\end{Theorem}

\begin{proof}
	Write $x = v \pi^\lambda$ and $y = \pi^\mu w $ for $v,w \in W$ and $\lambda, \mu \in \cT$. The assumption $\lambda, \nu \in \cT$ implies that there are finitely many $\eta\in \Delta^+$ such that $\innprod{\eta }{\mu}<0$ or $\innprod{\eta}{-\lambda }>0.$ Since $\Inv(v),$ $\Inv(w)$ are also finite, by \cite[Equation 24]{Muthiah-Orr-2019} we may write
	\begin{align}
	\label{eq:d:14}
	\Inv( v \pi^{\lambda}) = \cF_x \cup \left\{  \eta[m] \suchthat \eta \in \Delta^+,  \langle - \lambda, \eta \rangle \leq m < 0, \eta \notin \Inv(v) \right\}
	\end{align}
	and 
	\begin{align}
	\label{eq:d:17}
	\Inv( w^{-1} \pi^{-\mu}) = \cF_y \cup \left\{  \eta[m] \suchthat \eta \in \Delta^+,   0 \leq m < \langle \mu, \eta \rangle, \eta \notin \Inv(w^{-1}) \right\}
	\end{align}
	where $\cF_x$ and $\cF_y$ are finite sets. The theorem follows by comparing the right-hand sides of \eqref{eq:d:14} and \eqref{eq:d:17} to see $\Inv(x)\setminus\cF_x$ and $\Inv(y^{-1})\setminus \cF_y$ are disjoint. 
\end{proof}

\subsection{Statement and outline of the proof}

We are now ready to state the theorem. 

\begin{Theorem}
\label{thm:length-deficit-for-tits-cone-weyl-groups}
          Let $x,y \in W_\cT$.
  \begin{equation}
    \label{eq:length-deficit-kac-moody-tits}
    \ell(x) + \ell(y) = \ell(xy) + 2 | \Inv(x) \cap \Inv(y^{-1}) |
  \end{equation}
\end{Theorem}

The key portion of the proof is Proposition \ref{prop:thekeyprop} in \S \ref{ssect:keypart}. This is a statement involving a type of set defined in \eqref{eq:def:Inv++}, and one that appears in \cite[Theorem 3.14]{Muthiah-Orr-2019}. Indeed the proof uses techniques very similar to those in {\emph{op. cit.}} to carefully analyse the relationship of inversion sets and adapt the proof of Proposition \ref{prop:length-deficit-for-coxeter-weyl-groups} to the Kac-Moody affine setting. 

We explain the structure of this section. The proof of Theorem \ref{thm:length-deficit-for-tits-cone-weyl-groups} is by induction. We verify the statement directly for a large class of pairs $x,y \in W_{\cT}$ in \S \ref{ssect-basecases}. The induction is on the length sum of two (Coxeter Weyl) elements associated to the pair $x,y$. This is outlined in \S \ref{ssect:induction-layout}. One seemingly has to distinguish between cases, which are summarised in \S \ref{ssect:ascentdescentcases}. However, a standard symmetry argument (see \S \ref{ssect:symmetry}) allows one to reduce to just a single case. Finally, the proof of Proposition \ref{prop:thekeyprop} in \ref{ssect:keypart} contains the core of the argument.

\subsection{Base cases}\label{ssect-basecases}

The statement of Theorem \ref{thm:length-deficit-for-tits-cone-weyl-groups} can be directly checked for pairs $x=\pi^{\lambda}, y=\pi^{\mu }w \in W_\cT$ such that $\lambda ,\mu$ lie in the same Weyl chamber. In this case, their inversion sets intersect only in $\Delta ^+,$  i.e. only in single-affine roots. 

\begin{Lemma}\label{lem:simpleintersect}
Let $x=\pi^{\lambda}, y=\pi^{\mu }w \in W_\cT$ be elements such that $\lambda$ and $\mu$ lie in the same Weyl chamber. Then 
\begin{equation}\label{eq:simpleintersectinvset}
\Inv(x) \cap \Inv(y^{-1})=\{\eta \in \Delta ^+\mid \eta \in \Inv(w^{-1}), \langle \eta, \lambda \rangle <0, \langle \eta, \mu \rangle =0\}
\end{equation}
\end{Lemma}
\begin{proof}
Assume $\eta +m\pi \in \Inv(x) \cap \Inv(y^{-1}).$ We have 
\begin{equation}
\begin{split}
x(\eta +m\pi )= & \eta +(m+\langle \eta, \lambda \rangle)\pi <0\\
y^{-1}(\eta +m\pi )= & w^{-1}(\eta) +(m-\langle \eta, \mu \rangle)\pi<0. 
\end{split}
\end{equation}
By hypothesis $u(\lambda), u(\mu) \in P_{+}$ for some $u \in W$. Note $\langle \eta, \lambda \rangle=\langle u(\eta), u(\lambda )\rangle$ and $\langle \eta, \mu \rangle=\langle u(\eta), u(\mu )\rangle.$ Since $\eta +m\pi \in \tilde{\Delta}^+,$ $u(\eta )>0$ would imply $m+\langle \eta, \lambda \rangle\geq m$ and thus  $x(\eta +m\pi )\in \tilde{\Delta}^+.$ 
So $u(\eta )<0,$ whence $m-\langle \eta, \mu \rangle\geq m.$ It follows that $m=\langle \eta, \mu \rangle=0,$ $w^{-1}(\eta )<0,$ $\eta >0$ and $\langle \eta, \lambda \rangle<0.$
\end{proof}

Using this information, the statement of Theorem \ref{thm:length-deficit-for-tits-cone-weyl-groups} can now be checked for these pairs. 
        
\begin{Proposition}\label{prop:startcase-samechamber-onetransl}
Let $x=\pi^{\lambda}, y=\pi^{\mu }w \in W_\cT$ be elements such that $\lambda$ and $\mu$ lie in the same Weyl chamber.
  Then $\ell(x) + \ell(y) = \ell(xy) + 2 | \Inv(x) \cap \Inv(y^{-1}) |.$
\end{Proposition}
\begin{proof}
By Lemma \ref{lem:simpleintersect} it suffices to show that:
\begin{align}
  \label{eq:d:4}
\ell(x) + \ell(y) - \ell(xy)=2|\{\eta \in \Delta ^+\mid \eta \in \Inv(w^{-1}), \langle \eta, \lambda \rangle <0, \langle \eta, \mu \rangle =0\}|
\end{align}
This follows from the length formula \cite[(19)]{Muthiah-Orr-2019}, recalled in \eqref{eq:def-length-doubleaffine} above. We have $xy=\pi^{\lambda+\mu }w,$ so it suffices to check that the contributions of an $\eta \in \Delta ^+$ to $\ell(x),$ $\ell(y)$ and $\ell(xy)$ cancel, except in the case $\eta\in \Inv(w^{-1}), \langle \eta, \lambda \rangle <0, \langle \eta, \mu \rangle =0,$ when the total contribution is $2$.

Indeed this is straightforward to check, using the fact that by hypothesis $u(\lambda), u(\mu) \in P_{+}$ for some $u \in W$, and so $\langle \lambda, \eta \rangle \geq 0$ and $\langle \mu, \eta \rangle \geq 0$ for $\eta \notin \Inv(u),$ whereas $\innprod{\lambda }{\eta }\leq 0$ and $\innprod{\mu}{\eta }\leq 0$ if $\eta \in \Inv(u).$
\end{proof}

\subsection{Induction by length}\label{ssect:induction-layout}

Let $x,y\in \WTits$ be an arbitrary pair of elements. Write $x=\pi^\lambda w_x$ with $\lambda \in \Titscone$ and $w_x\in W.$ We may choose $w_y,w\in W$ so that $y=w_y\pi^{\mu}w,$ and $\mu $ lies in the same Weyl chamber as $\lambda .$ The proof of Theorem \ref{thm:length-deficit-for-tits-cone-weyl-groups} is by induction on $\ell(w_x)+\ell(w_y).$ The base case, $\ell(w_x)=\ell(w_y)=0$ is exactly Proposition \ref{prop:startcase-samechamber-onetransl}. The induction step is formulated in the following.

\begin{Claim}
	\label{claim:inductions-step-for-length-deficit}
	Let $x,y\in \WTits$  and write $x=\pi^\lambda w_x,$ $y=w_y\pi^{\mu}w,$ where $w_x,w_y,w\in W$ and $\lambda, \mu $ are in the same Weyl chamber. 
	Assume that Theorem \ref{thm:length-deficit-for-tits-cone-weyl-groups} holds for the pair $(x,y)$. 
	
	Let $s \in S$ be a simple reflection. If $w_xs > w_x$ then Theorem \ref{thm:length-deficit-for-tits-cone-weyl-groups} holds for the pair $(xs,y)$. If $sw_y > w_y$ then Theorem \ref{thm:length-deficit-for-tits-cone-weyl-groups} holds for the pair $(x,sy)$.
\end{Claim}

The proof of Theorem \ref{thm:length-deficit-for-tits-cone-weyl-groups} is now reduced to the proof of Claim \ref{claim:inductions-step-for-length-deficit}.

\subsubsection{Cases in the induction step}\label{ssect:ascentdescentcases}

We break the setting of Claim \ref{claim:inductions-step-for-length-deficit} into cases by comparing the pairs $x$ and $xs,$ $y$ and $sy,$ $xy$ and $xsy$ in the Bruhat order. The following Proposition gives identities that together imply Claim \ref{claim:inductions-step-for-length-deficit} in all of the resulting eight cases. Recall the notation defined in \eqref{eq:def:Inv++}.

\begin{Lemma}
  \label{lem:eight-cases-for-interesections-of-inversion-sets}
Let $x,y \in W_{\cT}$, $s \in S$. If
\begin{enumerate}
\item $x<xs,$ $y<sy,$ $xy<xsy$, then:
\begin{align}
\label{eq:case1-a}  2|\Inv(xs)\cap\Inv(y^{-1})|-2|\Inv(x)\cap\Inv(y^{-1})|& =1-|\Inv^{++}_{xy}(y^{-1}sy)| \\
\label{eq:case1-b}  2|\Inv(x)\cap\Inv((sy)^{-1})|-2|\Inv(x)\cap\Inv(y^{-1})|& = 1-|\Inv^{++}_{xy}(y^{-1}sy)|
\end{align}
\item $x>xs,$ $y>sy,$ $xy<xsy$, then:
\begin{align}
\label{eq:case2-a} 2|\Inv(xs)\cap\Inv(y^{-1})|-2|\Inv(x)\cap\Inv(y^{-1})|& = -1-|\Inv^{++}_{xy}(y^{-1}sy)| \\
\label{eq:case2-b} 2|\Inv(x)\cap\Inv((sy)^{-1})|-2|\Inv(x)\cap\Inv(y^{-1})|& =-1-|\Inv^{++}_{xy}(y^{-1}sy)|
\end{align}
\item $x<xs,$ $y>sy,$ $xy>xsy$, then:  
\begin{align}
\label{eq:case3-a} 2|\Inv(xs)\cap\Inv(y^{-1})|-2|\Inv(x)\cap\Inv(y^{-1})|& = 1+|\Inv^{++}_{xsy}(y^{-1}sy)| \\ 
\label{eq:case3-b} 2|\Inv(x)\cap\Inv((sy)^{-1})|-2|\Inv(x)\cap\Inv(y^{-1})|& =-1+|\Inv^{++}_{xsy}(y^{-1}sy)|
\end{align}
\item $x>xs,$ $y<sy,$ $xy>xsy$, then: 
\begin{align}
\label{eq:case4-a} 2|\Inv(xs)\cap\Inv(y^{-1})|-2|\Inv(x)\cap\Inv(y^{-1})|& = -1+|\Inv^{++}_{xsy}(y^{-1}sy)|\\
\label{eq:case4-b} 2|\Inv(x)\cap\Inv((sy)^{-1})|-2|\Inv(x)\cap\Inv(y^{-1})|& =1+|\Inv^{++}_{xsy}(y^{-1}sy)|
\end{align}
\end{enumerate}
\end{Lemma}

Lemma \ref{lem:eight-cases-for-interesections-of-inversion-sets} implies Claim \ref{claim:inductions-step-for-length-deficit} and therefore Theorem \ref{thm:length-deficit-for-tits-cone-weyl-groups} in turn, as we now explain.

\begin{proof}[Proof of Claim \ref{claim:inductions-step-for-length-deficit} assuming Lemma \ref{lem:eight-cases-for-interesections-of-inversion-sets}]
	First observe that by \eqref{eq:def-Bruhat-doubleaffine} the cases (1)-(4) are the only ones that occur. 
	The key is \cite[Theorem 3.4]{Muthiah-Orr-2019}, by which we have that 
	\begin{equation}\label{eq:MOthm-3.14}
	\ell(u\cdot s_{\beta [n]})-\ell(u)=|\Inv^{++}_u(s_{\beta [n]})|
	\end{equation}
	for any pair $u\in \WTits$ and $\beta [n]\in \afDp$ such that $u<u\cdot s_{\beta [n]}$ in the Bruhat order. (Note if $\beta[n]\in S$ then $\Inv^{++}_u(s_{\beta [n]})=\{s_{\beta [n]}\}$ and the statement is familiar.)
	
	Let $x,y$ be as in the statement of Claim \ref{claim:inductions-step-for-length-deficit}, and let $s \in S$ be a simple reflection. We may apply \eqref{eq:MOthm-3.14} with the smaller of $xy$ and $xsy$ in place of $u,$ and $s_{\beta [n]}=y^{-1}sy.$
	
	If $w_x s > w_x$ then combining this with \ref{eq:case1-a}, \ref{eq:case2-a}, \ref{eq:case3-a} or \ref{eq:case4-a}, respectively, we have: 
	\begin{equation}\label{eq:leftjump}
	(\ell(xs)+\ell(y)-\ell(xsy))-(\ell(x)+\ell(y)-\ell(xy))= 2|\Inv(xs)\cap\Inv(y^{-1})|-2|\Inv(x)\cap\Inv(y^{-1})|.
	\end{equation}
	If $s w_y > w_y$ then combining the above use of \eqref{eq:MOthm-3.14} with \ref{eq:case1-b}, \ref{eq:case2-b}, \ref{eq:case3-b}, or \ref{eq:case4-b}, we have: 
	\begin{equation}\label{eq:rightjump}
	(\ell(x)+\ell(sy)-\ell(xsy))-(\ell(x)+\ell(y)-\ell(xy))=2|\Inv(x)\cap\Inv((sy)^{-1})|-2|\Inv(x)\cap\Inv(y^{-1})|.
	\end{equation}
	This completes the proof of Claim \ref{claim:inductions-step-for-length-deficit}, i.e. the induction step, assuming the truth of Lemma \ref{lem:eight-cases-for-interesections-of-inversion-sets}.
\end{proof}

\subsubsection{Symmetry of cases}\label{ssect:symmetry}

The identities \eqref{eq:case1-a}-\eqref{eq:case4-b} are all of a similar flavour. We may in fact reduce the proof of all of them to just \eqref{eq:case1-a}. This is the content of the following. 

\begin{Lemma}\label{lem:cases-symmetry}
Lemma \ref{lem:eight-cases-for-interesections-of-inversion-sets} follows from \eqref{eq:case1-a}. That is, if for any pair $x,y\in \WTits$ and $s\in S$ simple reflection such that $x<xs,$ $y<sy$ and $xy<xsy$ we have (cf. \eqref{eq:case1-a})
\begin{equation}\label{eq:case1-a-repeat}  
2|\Inv(xs)\cap\Inv(y^{-1})|-2|\Inv(x)\cap\Inv(y^{-1})| =1-|\Inv^{++}_{xy}(y^{-1}sy)|,
\end{equation}
then all identities of Lemma \eqref{lem:eight-cases-for-interesections-of-inversion-sets} are true. 
\end{Lemma}

\begin{proof}

	Note that \eqref{eq:case3-a} follows from \eqref{eq:case2-a} and \eqref{eq:case4-a} follows from \eqref{eq:case1-a} by replacing $x$ by $xs$ (and leaving $y$ unchanged). Similarly, \eqref{eq:case3-b} follows from \eqref{eq:case1-b} and \eqref{eq:case4-b} follows from \eqref{eq:case2-b} by replacing $y$ by $sy$ (and leaving $x$ unchanged). So it suffices to consider cases (1) and (2). 
	
	Furthermore, we claim replacing both $x$ by $xs$ and $y$ by $sy$ reduces case (2) to case (1). Indeed, observe that if we write $x'=xs$ and $y'=sy$ then $x'y'=xy$ and $(y')^{-1}sy'=y^{-1}sy.$ Assume the pair $(x,y)$ satisfy the conditions in case (1). Then $(x',y')$ fall in case (2), and \eqref{eq:case2-a} and \eqref{eq:case2-b} (for $x'$ and $y'$) follow from \eqref{eq:case1-a} and \eqref{eq:case1-b} if we show 
	\begin{equation}\label{eq:(2)to(1)reduction}
	|\Inv(x)\cap \Inv(y^{-1})|+1=|\Inv(xs)\cap \Inv((sy)^{-1})|.
	\end{equation}
	Let $\alpha $ be the simple root corresponding to $s.$ Then \eqref{eq:(2)to(1)reduction} follows from:
	\begin{equation}\label{eq:(2)to(1)reduction-setversion}
	s(\Inv(x)\cap \Inv(y^{-1}))\cup \{\alpha \}=\Inv(xs)\cap \Inv((sy)^{-1}).
	\end{equation}
	Note that the assumption $x<xs,$ $y<sy$ implies that $\alpha \notin \Inv(x)\cup \Inv(y^{-1}).$ 
	
	We may therefore assume that $x,y$ are as in case (1), i.e. $x > xs$, $y < sy$ and $xy < xsy$. Then \eqref{eq:case1-b} follows from \eqref{eq:case1-a} if we show that 
	\begin{equation}\label{eq:(1b)to(1a)reduction}
	|\Inv(x)\cap\Inv((sy)^{-1})|=|\Inv(xs)\cap\Inv(y^{-1})|.
	\end{equation}
	These two finite sets are indeed in bijection by the reflection $s.$

\end{proof}

We have now reduced the proof of Theorem \ref{thm:length-deficit-for-tits-cone-weyl-groups} to just the proof of \eqref{eq:case1-a}. 

\subsection{The proof of Theorem \ref{thm:length-deficit-for-tits-cone-weyl-groups}}\label{ssect:keypart}

As explained in the preceeding subsections, it remains to prove the identity in \eqref{eq:case1-a}. We restate this here. 
\begin{Proposition}\label{prop:thekeyprop}
Let $x,y \in W_{\cT}$ and let $s \in S$ be a simple reflection so that $x<xs,$ $y<sy$ and $xy<xsy.$ Then
	\begin{equation}
	\label{eq:case1-a-third}  2|\Inv(xs)\cap\Inv(y^{-1})|-2|\Inv(x)\cap\Inv(y^{-1})| =1-|\Inv^{++}_{xy}(y^{-1}sy)| 
	\end{equation}
\end{Proposition}
\begin{proof}
	
The bijection constructed in the proof of Proposition \ref{prop:length-deficit-for-coxeter-weyl-groups} works exactly as before, and we get a bijection
\begin{equation}
  \label{eq:d:3}
 F:  \Inv(x) \backslash \Inv(y^{-1})  \sqcup \Inv(y^{-1}) \backslash \Inv(x) \rightarrow \Inv((xy)^{-1})
\end{equation}
by the same formula. We can reframe this as a bijection
\begin{equation}
  \label{eq:d:4}
  G_{x,y} : \Inv( (xy)^{-1} ) \sqcup ( \Inv(x) \cap \Inv(y^{-1}) ) \sqcup ( \Inv(y^{-1}) \cap \Inv(x) ) \overset{\sim}{\rightarrow} \Inv(x) \sqcup \Inv(y^{-1})
\end{equation}
where by $\sqcup$ we mean \emph{formal} disjoint union. Explicitly, we have
\begin{align}
  \label{eq:d:11}
  \text{For } \zeta \in \Inv((xy)^{-1}):\quad  & G_{x,y}(\zeta) =
  \begin{cases}
    -x^{-1}(\zeta) \in \Inv(x) \text{ if } x^{-1}(\zeta) < 0 \\
    x^{-1}(\zeta) \in \Inv(y^{-1}) \text{ if } x^{-1}(\zeta) > 0
  \end{cases} \\
\text{For } \beta \in( \Inv(x) \cap \Inv(y^{-1}) ) :\quad  & G_{x,y}(\beta) = \beta \in \Inv(x) \\
\text{For } \gamma \in( \Inv(y^{-1}) \cap \Inv(x) ) :\quad  & G_{x,y}(\gamma) = \gamma \in \Inv(y^{-1}) 
\end{align}

Because we are in the case where $xs > x$, we have an injection $H: \Inv(x) \hookrightarrow \Inv(xs)$ given by 
\begin{align}
  \label{eq:d:12}
 \text{For } \beta \in \Inv(x):\quad & H(\beta) = s(\beta)
\end{align}
The image of this injection is everything except $\alpha$ (note $xs(\alpha) = -x(\alpha) < 0$ by our assumption) where $\alpha$ is the simple positive root corresponding to $s$. We therefore have an injection
\begin{equation}
  \label{eq:d:5}
  H \sqcup \id : \Inv(x) \sqcup \Inv(y^{-1}) \hookrightarrow \Inv(xs) \sqcup \Inv(y^{-1}) 
\end{equation}
whose image misses only $\alpha \in \Inv(xs)$.

Also, as in \eqref{eq:d:4}, we have a bijection:
\begin{align}
  \label{eq:d:13}
  G_{xs,y}^{-1} : \Inv(xs) \sqcup \Inv(y^{-1}) \overset{\sim}{\rightarrow} \Inv( (xsy)^{-1} ) \sqcup ( \Inv(xs) \cap \Inv(y^{-1}) ) \sqcup ( \Inv(y^{-1}) \cap \Inv(xs) )
\end{align}
given by 
\begin{align}
  \text{For } \eta \in \Inv(xs):\quad  & G_{xs,y}^{-1}(\eta) =
  \begin{cases}
    -xs(\eta) \in \Inv((xsy)^{-1}) \text{ if } y^{-1}(\eta) > 0 \\
    \eta \in \Inv(xs) \cap \Inv(y^{-1}) \text{ if } y^{-1}(\eta)  < 0
  \end{cases} \\
  \text{For } \theta \in \Inv(y^{-1}):\quad  & G_{xs,y}^{-1}(\theta) =
  \begin{cases}
    xs(\theta) \in \Inv((xsy)^{-1}) \text{ if } xs(\theta) > 0 \\
    \theta \in \Inv(y^{-1}) \cap \Inv(xs) \text{ if } -xs(\theta)  < 0
  \end{cases} \\
\end{align}
We therefore have the injection
\begin{align*}
G_{xs,y}^{-1} \circ (H \sqcup \id)\circ G_{x,y} : \Inv( (xy)^{-1} ) \sqcup ( \Inv(x) \cap \Inv(y^{-1}) ) \sqcup ( \Inv(y^{-1}) \cap \Inv(x) )  \hookrightarrow \\ \Inv( (xsy)^{-1} ) \sqcup ( \Inv(xs) \cap \Inv(y^{-1}) ) \sqcup ( \Inv(y^{-1}) \cap \Inv(xs) )
\end{align*}
whose image misses only $G_{xs,y}^{-1}(\alpha)$

We further have a bijection (constructed in \cite[\S 4]{Muthiah-Orr-2019}):
\begin{equation}
  \label{eq:d:7}
  K : \Inv( (xsy)^{-1}) \overset{\sim}{\rightarrow} \Inv((xy)^{-1}) \sqcup  \Inv^{++}_{xy}(y^{-1}sy)
\end{equation}
Explicitly \cite[Corollary 4.4]{Muthiah-Orr-2019}, we have: 
\begin{align}
  \label{eq:d:15}
  K(\xi) =
  \begin{cases}
    \xi \in \Inv((xy)^{-1}) &\text{ if } (xy)^{-1}(\xi) < 0 \\
    xsx^{-1}(\xi) \in \Inv((xy)^{-1}) &\text{ if } (xy)^{-1}(\xi) > 0 \text{ and } xsx^{-1}(\xi) > 0\\
    (xy)^{-1}(\xi) &\text{ if } (xy)^{-1}(\xi) > 0 \text{ and } xsx^{-1}(\xi) < 0 \\
  \end{cases}
\end{align}
So we get an injection
\begin{align*}
  L = (K \sqcup \id \sqcup \id) \circ   G_{xs,y}^{-1} \circ (H \sqcup \id)\circ G_{x,y} : \Inv( (xy)^{-1} ) \sqcup ( \Inv(x) \cap \Inv(y^{-1}) ) \sqcup ( \Inv(y^{-1}) \cap \Inv(x) ) \hookrightarrow \\
  \Inv((xy)^{-1}) \sqcup  \Inv^{++}_{xy}(y^{-1}sy) \sqcup ( \Inv(xs) \cap \Inv(y^{-1}) ) \sqcup ( \Inv(y^{-1}) \cap \Inv(xs) )
\end{align*}
whose image misses only $(K \sqcup \id \sqcup \id) \circ G_{xs,y}^{-1}(\alpha)$.

Let us analyze $L$ explicitly. Let
\begin{align}
  \label{eq:d:16}
  \cS = \left\{  \zeta \in \Inv( (xy)^{-1}) \suchthat L(\zeta) \in \Inv( (xy)^{-1}) \right\}
\end{align}
We know $\Inv^{++}_{xy}(y^{-1}sy)$ is finite by \cite[Theorem 5.2]{Muthiah-Orr-2019}, and the intersections $ \Inv(x) \cap \Inv(y^{-1}) $ and $ \Inv(xs) \cap \Inv(y^{-1}) $ are finite by Theorem \ref{thm:finiteness-of-invx-intersect-invyinverse}. So we know that the complement of $\cS$ in $\Inv( (xy)^{-1})$ is a finite set.

The restricted map $L : \cS \rightarrow \Inv( (xy)^{-1})$ is given by the following formula:
\begin{align}
  \label{eq:d:6}
  \text{For } \zeta \in \cS:\quad L(\zeta) =
   \begin{cases}
     \zeta \text{ if } x^{-1}(\zeta) < 0 \\
     xsx^{-1}(\zeta) \text{ if } x^{-1}(\zeta) > 0 \text{ and } (xsy)^{-1}(\zeta) < 0 \\
     \zeta \text{ if } \text{ if } x^{-1}(\zeta) > 0 \text{ and } (xsy)^{-1}(\zeta) > 0 
   \end{cases}
\end{align}
We claim that the map $L$ induces a bijection:
  \begin{align}
    \label{eq:d:9}
    L : \cS \overset{\sim}{\rightarrow} \cS
  \end{align}
Let $\cS_1 = \{ \zeta \in \cS \suchthat L(\zeta) = \zeta) \}$, and let $\cS_2 = \{ \zeta \in S \suchthat L(\zeta) = xsx^{-1}(\zeta) \}$. It is clear that $L$ induces a bijection of $\cS_1$. We will show $L$ induces a bijection of $\cS_2$. Let $\zeta \in \cS_2$. Then $x^{-1}(\zeta) > 0$ and $(xsy)^{-1}(\zeta) < 0$, and $L(\zeta) = xsx^{-1}(\zeta)$. We will show that $L(\zeta)$ lies in $\cS_2$. First we compute $G_{x,y}(L(\zeta))$. We have $x^{-1}(L(\zeta)) = sx^{-1}(\zeta)$. Because $x^{-1}(\zeta) > 0$, this can only be negative if $\zeta$ were equal to $x(\alpha)$. A straightforward computation shows that $x(\alpha) \notin \cS$, so $\zeta$ cannot be equal to $x(\alpha)$.

So  $G_{x,y}(L(\zeta)) = x^{-1}( L(\zeta) = s x^{-1}(\zeta) \in \Inv(y^{-1})$. Next we compute $G_{xs,y}^{-1}(\theta) $ for $\theta =  s x^{-1}(\zeta))$. We have $xs(\theta) = \zeta > 0$, so $G_{xs,y}^{-1}(\theta) = \zeta$. Finally, because $\zeta \in \Inv((xy)^{-1})$, $K(\zeta) = \zeta$. So we have shown that $L(\zeta) \in \cS_2$, and $L^2(\zeta) = \zeta$. We conlude that $L$ induces a bijection of $\cS$ .

Looking at the complement of $\cS$, we see that $L$ induces an injection between 
\begin{align}
  \label{eq:d:8}
  L:   (\Inv( (xy)^{-1} ) \backslash \cS) \sqcup ( \Inv(x) \cap \Inv(y^{-1}) ) \sqcup ( \Inv(y^{-1}) \cap \Inv(x) ) \hookrightarrow \\
  (\Inv((xy)^{-1}) \backslash \cS) \sqcup  \Inv^{++}_{xy}(y^{-1}sy) \sqcup ( \Inv(xs) \cap \Inv(y^{-1}) ) \sqcup ( \Inv(y^{-1}) \cap \Inv(xs) )
\end{align}
whose image misses only one element. Observe now that the source and target of \ref{eq:d:8} are now finite sets. So we can conclude that:
\begin{align}
    \label{eq:d:10}
    2 | \Inv(x) \cap \Inv(y^{-1}) )| + 1 = |\Inv^{++}_{xy}(y^{-1}sy)| + 2|\Inv(xs) \cap \Inv(y^{-1})| 
\end{align}
  
\end{proof}

%%% Local Variables:
%%% mode: latex
%%% TeX-master: "main"
%%% End:

\section{The main conjecture and examples in the Kac-Moody affine setting}
\label{sec:conjecture-in-the-kac-moody-affine-setting}

Theorem \ref{thm:length-deficit-for-tits-cone-weyl-groups} establishes Proposition \ref{prop:length-deficit-for-coxeter-weyl-groups} in the Kac-Moody affine setting. We need this to state the main conjecture of this paper.

\begin{Conjecture}
  \label{conj:main-conjecture}
Theorem \ref{thm:expansion-properties-Coxeter} holds in the Kac-Moody affine case. Additionally, the explicit formulas in the length-deficit $0$ and $2$ cases (Examples \ref{example:lengthdeficit-0-Coxeter} and \ref{example:lengthdeficit-1-Coxeter}) hold in the Kac-Moody affine case.  
\end{Conjecture}

The rest of this section is devoted to explicit computations that confirm this conjecture for specific examples in the case of $\widehat{SL_2}$.  These computations illustrate the multiplication algorithm described in \S \ref{sssect:mult-in-T-basis}. 

Working in the Kac-Moody root system $\widehat{SL_2}$, consider the cocharacter $\Lambda_0$, which satisfies $\langle \alpha_0, \Lambda_0  \rangle = 1$, and $\langle \alpha_1, \Lambda_0  \rangle = 0$. This does not uniquely determine $\Lambda_0$ because $\widehat{SL_2}$ has a singular Cartan matrix. Recall that in this case, we also have a choice for the character $\rho$. We fix the choices for $\Lambda_0$ and $\rho$ so that $\langle \rho, \Lambda_0  \rangle = 0$.

\subsection{Length deficit zero}

Pairs of the form $\pi^{\mu_1}$ and $\pi^{\mu_2}$ where $\mu_1$ and $\mu_2$ lie in the same Weyl chamber are easily seen to be length additive, and in these cases \eqref{eq:d:1} and \eqref{eq:d:2} immediately imply
\begin{equation}
  \label{eq:easy-length-additivity}
  T_{\pi^{\mu_1}} T_{\pi^{\mu_2}} = T_{\pi^{\mu_1 + \mu_2}}
\end{equation}

However, there are other length additive pairs not of this form. For example, let $x=\pi^{\Lo+4\delta }$ and $y=s_1s_0\pi^{\Lo+\delta }$. Then $\ell(x)=16$, $\ell(y)=2$ and $\ell(x)+\ell(y)-\ell(xy)=0$. 
We may write $T_x$ and $T_y$ in the Bernstein basis, and use the Bernstein relation \eqref{eq:Bernstein} and the generalized Iwahori-Matsumoto relations \eqref{eq:Ts_onright}, \eqref{eq:Ts_onleft} to compute: 
\begin{equation}
\begin{split}
T_xT_y=& q^{\innprod{\rho}{2\Lo+5\delta }}\cdot \Th{{\Lo+4\delta }}T_{s_1}^{-1}T_{s_0}^{-1}\Theta^{{\Lo+\delta }}\\
=& q^{\innprod{\rho}{2\Lo+5\delta }-1}\cdot T_{s_1}^{-1}\Th{{\Lo+4\delta }}\Th{\Lo+\alpha_1}T_{s_0}\\
=& T_{s_1}^{-1}T_{\pi^{2\Lo+4\delta +\alpha_1}}T_{s_0}= T_{\pi^{\Lo+4\delta }s_1s_0\pi^{\Lo+\delta }}= T_{xy}
\end{split}
\end{equation}

\subsection{Length deficit two}

Let $x=s_0\pi^{\Lo}=\pi^{\Lo-\alpha_0}s_0$ and $y=s_0s_1s_0\pi^{\Lo}.$ We check 
that $\Inv(x)\cap\Inv(y^{-1})=\{\beta \}$ where $\beta =-2\alpha_0-\alpha_1+\pi =\alpha_1-2\delta +\pi =(-\alpha_1+2\delta)[-1]=(2\alpha_0+\alpha_1)[-1].$ Therefore, by Theorem \ref{thm:length-deficit-for-tits-cone-weyl-groups}, $x$ and $y$ are a length-deficit two pair. 
	
	Using the generalized Iwahori-Matsumoto relations \eqref{eq:Ts_onleft} 
        we have $T_x=T_{s_0}^{-1}T_{\pi^{\Lo}}$ and 
        $T_y=T_{s_0}^{-1}T_{s_1}^{-1}T_{s_0}^{-1}T_{\pi^{\Lo}}.$ Since $\Lo $ is dominant, we have $T_{\pi^{\Lo}}=q^{\innprod{\Lo}{\rho}}\cdot \Th{\Lo}.$
	
	Therefore:
	\begin{equation}\label{eq:IMcomplete}
	T_x\cdot T_y=q^{2\innprod{\Lo}{\rho}}\cdot T_{s_0}^{-1}\Th{\Lo}T_{s_0}^{-1}T_{s_1}^{-1}T_{s_0}^{-1}\Th{\Lo}
	\end{equation}

        Using the Bernstein relations, one has:
	\begin{equation}\label{eq:move_s0_oneover}
	\begin{split}
	T_{s_0}^{-1}\Th{\Lo} 
	= & q^{-1}\Th{\Lo-\alpha_0}T_{s_0}\\
	\end{split}
	\end{equation}
	Substituting into the right-hand side of \eqref{eq:IMcomplete}:
	\begin{equation}\label{eq:Bern_mult_step1}
	\begin{split}
	T_x\cdot T_y
	=& q^{2\innprod{\Lo}{\rho}-2}\cdot \Th{\Lo-\alpha_0}T_{s_1}^{-1}\Th{\Lo-\alpha_0}T_{s_0}\\
	\end{split}
	\end{equation}
	Applying \eqref{eq:T-identities} and \eqref{eq:Bernstein} again, noting $\innprod{\alpha_1}{\Lo-\alpha_0}=2$ it is straightforward to check that:
	\begin{equation}
	\begin{split}
	T_{s_1}^{-1}\Th{\Lo-\alpha_0}
	=& q^{-1}\Th{\Lo-\alpha_0-2\alpha_1}T_{s_1}+q^{-1}(q-1)\Th{\Lo-\alpha_0-\alpha_1}\\
	\end{split}
	\end{equation}
	Again substituting into \eqref{eq:Bern_mult_step1} one then has: 
	\begin{equation}\label{eq:Bern_mult_step2}
	\begin{split}
	T_x\cdot T_y
	=& q^{2\innprod{\Lo}{\rho}-3}\cdot \Th{\Lo-\alpha_0}(\Th{\Lo-\alpha_0-2\alpha_1}T_{s_1}+(q-1)\Th{\Lo-\alpha_0-\alpha_1})T_{s_0}\\
	=& q\cdot T_{\pi^{2\Lo-2\alpha_0-2\alpha_1}}T_{s_1}T_{s_0}+(q-1)T_{\pi^{2\Lo-2\alpha_0-\alpha_1}}T_{s_0}\\
	\end{split}
	\end{equation}
	The last line of the equation converts the elements of the Bernstein basis back to the $T$ basis. Note that the weights $2\Lo-2\alpha_0-2\alpha_1$ and $2\Lo-2\alpha_0-\alpha_1$ are dominant, and thus the BKP algorithm \S \ref{sssect:BKPalgorithm} is immediate.

	Observe that $xy=\pi^{2\Lo-2\delta }s_1s_0$ and $xs_{\beta }y=\pi^{2\Lo-\delta -\alpha_0}s_0$. Since we have $\pi^{2\Lo-2\alpha_0-2\alpha_1}(\alpha_1)=\alpha_1>0$ and $(\pi^{2\Lo-2\alpha_0-2\alpha_1}s_1)(\alpha_0)=\alpha_0+2\alpha_1+2\pi >0$ the first term is $T_{xy}$ and since $\pi^{2\Lo-2\alpha_0-\alpha_1}(\alpha_0)=\alpha_0>0$ the second term is $T_{xs_{\beta}y}$ by the generalized Iwahori-Matsumoto relations \ref{eq:IMcomplete} and therefore 
	\begin{equation}\label{eq:explicitcomp-result}
	\begin{split}
	T_x\cdot T_y=& q\cdot T_{\pi^{2\Lo-2\delta}s_1s_0}+(q-1)\cdot T_{\pi^{2\Lo-\delta -\alpha_0}s_0}\\
	=& q\cdot T_{xy}+(q-1)\cdot T_{xs_{\beta}y},
	\end{split}
	\end{equation}
	as predicted. 	

Note that $\beta \in \Inv(x) \cap \Inv(y^{-1})$, implies  $xs_\beta y > xy$.
        
\subsection{Length deficit four with three terms in the product}

In the case of Weyl groups, we saw that for length-deficit four pairs, the product in the Hecke algebra may have three or four terms. The same phenomenon occurs here.

Let $x=\pi^{\Lo+2\alpha_1}s_0s_1$ and $y= \pi^{\Lo-\alpha_1}s_1(s_0s_1)^2.$ Then $\ell(x)=14$ and $\ell(y)=5,$ and $\ell(x)+\ell(y)-\ell(xy)=4.$ We may check that $\Inv (x)\cap \Inv(y^{-1})=\{\alpha_0,\ \alpha_1[-1]\}.$ Writing $x=s_0s_1s_0\pi^{\Lo+4\delta }s_0$ and $y=s_1s_0\pi^{\Lo+\delta }s_1s_0s_1$ we may rewrite $T_x$ and $T_y$ in the Bernstein basis, and use three instances of the Bernstein relation to compute:
\begin{equation}
\begin{split}
T_xT_y=& q^{\innprod{\rho}{2\Lo+5\delta }}\cdot T_{s_0s_1s_0}^{-1}\Theta^{{\Lo+4\delta }}T_{s_0} \cdot  T_{s_1}^{-1}T_{s_0}^{-1}\Theta^{{\Lo+\delta }}T_{s_1s_0s_1}\\
=& q^{\innprod{\rho}{2\Lo+4\delta }}\cdot T_{s_0s_1s_0}^{-1}(T_{s_0}\Th{\Lo+4\delta -\alpha_0}+(q-1)\Th{\Lo+4\delta })\cdot \\ 
& \cdot (\Th{\Lo-\alpha_1}T_{s_1}+(q-1)\Th{\Lo})T_{(s_0s_1)^2}\\
=& q^{\innprod{\rho}{2\Lo+4\delta }}\cdot T_{s_0}^{-1}\Th{2\Lo+3\delta}T_{(s_0s_1)^2}+\\
&+q^{\innprod{\rho}{2\Lo+4\delta }}(q-1)\cdot T_{s_1s_0}^{-1}\Th{2\Lo+4\delta -\alpha_0}T_{(s_0s_1)^2}+\\
&+ q^{\innprod{\rho}{2\Lo+4\delta }}(q-1)\cdot T_{s_0s_1s_0}^{-1}\Th{2\Lo+4\delta -\alpha_1}T_{s_1}T_{(s_0s_1)^2}+\\
&+q^{\innprod{\rho}{2\Lo+4\delta }}(q-1)^2\cdot T_{s_0s_1s_0}^{-1}\Th{2\Lo+4\delta}T_{(s_0s_1)^2}\\
\end{split}
\end{equation}
The first, second and fourth terms have dominant weights $\mu$ appearing as $\Th{\mu },$ so we may convert these back to the $T$ basis directly. For the third term, we may apply a Bernstein relation again. After simplification, this yields: 
\begin{equation}
\begin{split}
T_xT_y=& q^{2}\cdot T_{s_0}^{-1}T_{\pi^{2\Lo+3\delta}}T_{(s_0s_1)^2}+\\
&+q(q-1)\cdot T_{s_1s_0}^{-1}T_{\pi^{2\Lo+3\delta +\alpha_1}}T_{(s_0s_1)^2}+\\
&+ (q-1)\cdot T_{(s_1s_0)^2}^{-1}T_{\pi^{2\Lo+4\delta +\alpha_1}}T_{(s_0s_1)^2}\\
\end{split}
\end{equation}
Now we may check Iwahori-Matsumoto relations (or compute lengths) to recognize each term is supported at a single element of the $T$-basis, namely: 
\begin{equation}
\begin{split}
T_xT_y=& q^{2}\cdot T_{s_0\pi^{2\Lo+3\delta}(s_0s_1)^2}+q(q-1)\cdot T_{s_0s_1\pi^{2\Lo+3\delta +\alpha_1}(s_0s_1)^2}+(q-1)\cdot T_{(s_0s_1)^2\pi^{2\Lo+4\delta +\alpha_1}(s_0s_1)^2}\\
=& q^{2}\cdot T_{xy}+q(q-1)\cdot T_{xs_{\alpha_1[-1]}y}+(q-1)\cdot T_{xs_{\alpha_0}y}.\\
\end{split}
\end{equation}
Observe that $\ell(xs_{\alpha_1[-1]}y)-\ell(xy)=1$ and $\ell(xs_{\alpha_0}y)-\ell(xy)=3.$ (Cf. Example \ref{example:lengthdeficit-2-Coxeter}.) As above, we know all terms are greater than or equal to $xy$ in the Bruhat order.

\subsection{Length deficit four with four terms in the product}

Let $x=\pi^{\Lo+2\alpha_1}$ and $y=\pi^{\Lo+\alpha_1}s_0s_1.$ Then $\ell(x)=16,$ $\ell(y)=4,$ and $\ell(x)+\ell(y)-\ell(xy)=4.$ We may check that $\Inv (x)\cap \Inv(y^{-1})=\{-\alpha_1+2\delta ,\ -\alpha_1+3\delta \}.$ Writing $x=\pi^{s_0s_1s_0(\Lo+4\delta )}$ and $y=\pi^{s_0(\Lo+\delta )}s_0s_1$ we may rewrite $T_x$ and $T_y$ in the Bernstein basis, use the Bernstein relation to perform the multiplication, and return the product into the $T$ basis as follows: 
\begin{equation}
\begin{split}
T_xT_y=&q^{\innprod{\rho}{2\Lo+5\delta }} T_{s_0s_1s_0}^{-1}\Theta_{\Lo+4\delta }T_{s_0}T_{s_1}\Theta_{\Lo+\delta }T_{s_1}\\
=&q^{\innprod{\rho}{2\Lo+5\delta }} T_{s_0s_1s_0}^{-1}\Theta_{2\Lo+4\delta +\alpha_1}T_{s_0}T_{s_1}^2+(q-1)q^{\innprod{\rho}{2\Lo+5\delta }} T_{s_0s_1s_0}^{-1}\Theta_{2\Lo+5\delta }T_{s_1}^2\\
=&q\cdot T_{s_0s_1s_0}^{-1}T_{\pi ^{2\Lo+4\delta +\alpha_1}}T_{s_0}T_{s_1}^2+(q-1)T_{s_0s_1s_0}^{-1}T_{\pi^{2\Lo+5\delta }}T_{s_1}^2\\
\end{split}
\end{equation}
Applying the quadratic relations and the generalized Iwahori-Matsumoto relations, this can be written as: 
\begin{equation}\label{eq:TxTy-Tbasis-final}
\begin{split}
T_xT_y=&q^2\cdot T_{xy}+q(q-1)\cdot T_{xs_{-\alpha_1+2\delta}y}+q(q-1)T_{xs_{-\alpha_1+3\delta}y} +(q-1)^2T_{(s_0s_1)^2\pi^{2\Lo+5\delta }}
\end{split}
\end{equation}
Here $\ell(xs_{-\alpha_1+2\delta}y)-\ell(xy)=1,$ $\ell(xs_{-\alpha_1+3\delta}y)-\ell(xy)=1,$ and $\ell((s_0s_1)^2\pi^{2\Lo+5\delta })-\ell(xy)=2.$

As in the previous examples, the first, second, and third term are immediately greater than or equal to $xy$ in the Bruhat order. However, we need to check the fourth term $(s_0s_1)^2\pi^{2\Lo+5\delta }=xy\cdot s_{\alpha_0[-1]}s_1.$ We have $xy(\alpha_0[-1])=\alpha_1-3\delta +\pi >0,$ (hence $\alpha_0[-1]\in \Inv^{++}_{xy}(\alpha_0[-1])$ and) $xy<xy\cdot s_{\alpha_0[-1]};$ similarly $xy\cdot s_{\alpha_0[-1]}(\alpha_1)=\pi ^{2\Lo-3\delta+4\alpha_1}s_0s_1s_0(\alpha_1)=-\alpha_1+4\delta >0,$ whence $xy\cdot s_{\alpha_0[-1]}<xy\cdot s_{\alpha_0[-1]}s_1.$ So the fourth element of the support is also greater than $xy$ in the Bruhat order.

%%% Local Variables:
%%% mode: latex
%%% TeX-master: "main"
%%% End:

\section{The $q=0$ specialization and the Demazure product}\label{sect:Demazure-product}

Let $W$ be a Weyl group (or more generally any Coxeter group). It follows directly from the braid and quadratic relations that the $T$-basis of the Hecke algebra $\cH_W$ has structure coefficients lying in $\ZZ[q]$. We can therefore consider the $q=0$ specialization of $\cH_W$. It is easy to check that in the resulting nil-Hecke algebra 
\begin{equation}
  \label{eq:classical-demazure-product-and-q-equal-zero}
  T_x T_y \equiv (-1)^{\ell(x) + \ell(y)} T_{x \star y}
\end{equation}
for all $x,y \in W,$ where $x \star y$ denotes the \emph{Demazure product} on $W$ \cite[p. 721]{kenney2014coxeter}. Recall that the Demazure product is an associative multiplication on $W$ characterized by
\begin{equation}
  \label{eq:d:2}
  w \star s =
  \begin{cases}
    w s \textif \ell(ws) = \ell(w) + 1 \\
    w \textif \ell(ws) = \ell(w) - 1
  \end{cases}
\end{equation}
for $s$ a simple reflection in $W$, and $w \in W$ arbitrary.

In the Kac-Moody affine case, we consider $W_\cT = W \ltimes \cT$, and the corresponding Kac-Moody affine Hecke algebra $\cH = \cH_{W_{\cT}}$. As mentioned in \S \ref{sssect:mult-in-T-basis} the structure coefficients of the $T$-basis lie in $\ZZ[q]$ as well, but this is a non-trivial result of \cite[\S 6.7]{BardyPanse-Gaussent-Rousseau-2016} and \cite[Theorem 3.22]{Muthiah-2018}. In particular, the proof is somewhat non-constructive: there is an algorithm to compute structure coefficients that lie in $\ZZ[q,q^{-1}]$. However, since the structure coefficients must be an integer when $q$ is specialized to any prime power, the structure coefficients must lie in $\ZZ[q]$. We can therefore make the following conjecture.

\begin{Conjecture}
  \label{conj:q-equals-zero-specialization}
Let $x,y \in W_{\cT}$. Then there exists a unique $x \star y \in W_{\cT}$ such that:
\begin{equation}
  \label{eq:d:22}
  T_x T_y \equiv (-1)^{\ell(x) + \ell(y)} T_{x \star y} \mod q
\end{equation}
\end{Conjecture}

We note that all the examples computed in Section \ref{sec:conjecture-in-the-kac-moody-affine-setting} above align with this conjecture. A positive answer to this conjecture would give a definition of the \emph{Kac-Moody affine Demazure product}. We mention that another approach is to generalize the work of Schremmer \cite{Schremmer} to the Kac-Moody affine setting. 

\subsection{Length deficit $0$ and $2$}

Conjecture \ref{conj:q-equals-zero-specialization} along with conjecture \ref{conj:main-conjecture} imply the following corollaries.

\begin{Conjecture}\label{conj:demprod-lengthadd}
  If $x,y \in W_{\cT}$, and $\ell(xy) = \ell(x) + \ell(y)$, then:
  \begin{equation}
    \label{eq:d:24}
    x \star y = xy
  \end{equation}
\end{Conjecture}

\begin{Conjecture}\label{conj:demprod-lengthdef2}
  If $x,y \in W_{\cT}$, and $\ell(xy) = \ell(x) + \ell(y)-2$, then by Theorem
\ref{thm:length-deficit-for-tits-cone-weyl-groups}, we have $\Inv(x) \cap \Inv(y^{-1}) = \{\beta[n]\}$ for a single Kac-Moody affine root $\beta[n]$. In this case, we have
  \begin{equation}
    \label{eq:d:24}
    x \star y = xs_{\beta[n]} y
  \end{equation}
\end{Conjecture}

\subsection{Bruhat order and Demazure product}

In the case where $W$ is a Coxeter Weyl group, then there is a close relationship between the Bruhat order and the Demazure product. For each $x \in W$, define the \emph{down set} of $x$ by:
\begin{equation}
  \label{eq:d:25}
  \downarrow\!{(x)} = \left\{ z \in W \suchthat z \leq x \right\}
\end{equation}
Then it is known that \cite[Proposition 8.]{kenney2014coxeter}:
\begin{equation}
  \label{eq:d:26}
  \downarrow\!{(x)} \cdot \downarrow\!{(y)} = \downarrow\!{(x \star y)} 
\end{equation}
Note on the left, $\downarrow\!{(x)} \cdot \downarrow\!{(y)}$ denotes the set of all products in the usual group multiplication on $W$. In particular, this shows that the Demazure product is characterized by the Bruhat order and the group multiplication on $W$. Both of these are available in the Kac-Moody affine case, so that leads us to ask the following question.

\begin{Question}
Does equation \eqref{eq:d:26} hold in the Kac-Moody affine case?
\end{Question}

In the Kac-Moody affine case, we currently do not have a good way to understand down sets, and we cannot verify \eqref{eq:d:26} in any particular example. This is our reason for framing this as a question rather than a conjecture.

%%% Local Variables:
%%% mode: latex
%%% TeX-master: "main"
%%% End:

\bibliographystyle{amsalpha}
\bibliography{references}

\providecommand{\bysame}{\leavevmode\hbox to3em{\hrulefill}\thinspace}
\providecommand{\MR}{\relax\ifhmode\unskip\space\fi MR }
\providecommand{\MRhref}[2]{%
  \href{http://www.ams.org/mathscinet-getitem?mr=#1}{#2}
}
\providecommand{\href}[2]{#2}
\begin{thebibliography}{BPGR16}

\bibitem[BKP16]{Braverman-Kazhdan-Patnaik}
Alexander Braverman, David Kazhdan, and Manish~M. Patnaik,
  \emph{Iwahori-{H}ecke algebras for {$p$}-adic loop groups}, Invent. Math.
  \textbf{204} (2016), no.~2, 347--442. \MR{3489701}

\bibitem[BPGR16]{BardyPanse-Gaussent-Rousseau-2016}
Nicole Bardy-Panse, St\'{e}phane Gaussent, and Guy Rousseau,
  \emph{Iwahori-{H}ecke algebras for {K}ac-{M}oody groups over local fields},
  Pacific J. Math. \textbf{285} (2016), no.~1, 1--61. \MR{3554242}

\bibitem[IM65]{iwahori1965some}
Nagayoshi Iwahori and Hideya Matsumoto, \emph{On some {B}ruhat decomposition
  and the structure of the {H}ecke rings of $ p $-adic {C}hevalley groups},
  Publications Math{\'e}matiques de l'IH{\'E}S \textbf{25} (1965), 5--48.

\bibitem[Ken14]{kenney2014coxeter}
Toby Kenney, \emph{{C}oxeter groups, {C}oxeter monoids and the {B}ruhat order},
  Journal of Algebraic Combinatorics \textbf{39} (2014), 719--731.

\bibitem[MO19]{Muthiah-Orr-2019}
Dinakar Muthiah and Daniel Orr, \emph{On the double-affine {B}ruhat order: the
  $\varepsilon =1$ conjecture and classification of covers in {ADE} type},
  Algebraic Combinatorics \textbf{2} (2019), no.~2, 197--216 (en).

\bibitem[Mut18]{Muthiah-2018}
Dinakar Muthiah, \emph{On {I}wahori-{H}ecke algebras for {$p$}-adic loop
  groups: double coset basis and {B}ruhat order}, Amer. J. Math. \textbf{140}
  (2018), no.~1, 221--244. \MR{3749194}

\bibitem[Sch24]{Schremmer}
Felix Schremmer, \emph{Affine {B}ruhat order and {D}emazure products}, Forum of
  Mathematics, Sigma \textbf{12} (2024), e53.

\end{thebibliography}

\end{document}